\def\version{14.06.2018, 14.42h }\def\users{final-layout}  %
\def\users{us}  
\documentclass[11pt]{article}
\usepackage{times,cite}
\usepackage[]{graphicx}
\usepackage{amsmath,amssymb,amsthm,mathrsfs,epsf,a4,color,graphicx,eucal}
\usepackage{amsfonts} 
\usepackage{upgreek}
\usepackage{comment} 

\usepackage{enumerate,esint} 

\usepackage{ifthen}
\usepackage[normalem]{ulem}
\usepackage{psfrag}
\usepackage{epsfig}

{
\usepackage{fancyhdr}
\pagestyle{fancy}
\headheight=28pt
\rhead{\color{green}
Dynamic perfect plasticity and damage...\\
by E.Davoli, T.Roub\'\i\v{c}ek, U.Stefanelli}
\chead{}
\lhead{\color{green}Version\,\version, file:\,\jobname
\\compiled:
\number\day.\number\month.\number\year\ at
\the\hour:\ifnum\minute<10 0\fi\the\minute\ h }
\definecolor{labelkey}{rgb}{1.,.2,0.}
}
\newcount\hour \newcount\minute
\hour=\time
\divide \hour by 60
\minute=\time
\loop \ifnum \minute > 59 \advance \minute by -60 \repeat

\definecolor{brown}{rgb}{0.5,0,0}
\ifthenelse{\equal{\users}{final-layout}}{

	\newcommand{\COMMENT}[1]{}
	\newcommand{\DELETE}[1]{}

        \newcommand{\REM}[1]{\marginpar{\bfseries\tiny{\color{blue}}}}
        
\newcommand\UUU{\color{black}}
\newcommand\EEE{\color{black}}

\newcommand{\eq}[1]{\eqref{#1}}
\newcommand{\QUE}[1]{\COMMENT{}}
\definecolor{DDGreen}{rgb}{0.,0.55,0.}
\newcommand\GGG{\color{black}}
        }
{
 
 \newcommand{\COMMENT}[1]{{\color{red}\uuline{#1}\color{black}}}
 \newcommand{\DELETE}[1]{{\color{brown}\sout{#1}\color{black}}}

 \newcommand{\REM}[1]{\marginpar{\bfseries\tiny{\color{blue}#1}}}
 
\newcommand\UUU{\color{black}}
\newcommand\EEE{\color{black}}

\newcommand{\eq}[1]{\eqref{#1}}
\newcommand{\QUE}[1]{\COMMENT{#1}}
\definecolor{DDGreen}{rgb}{0.,0.55,0.}
\newcommand\GGG{\color{black}}
}

\begin{document}









\allowdisplaybreaks
\theoremstyle{plain}
\newtheorem{theorem}{Theorem}[section]
\newtheorem{proposition}[theorem]{Proposition}
\newtheorem{lemma}[theorem]{Lemma}
\newtheorem{example}[theorem]{Example}
\newtheorem{corollary}[theorem]{Corollary}
\theoremstyle{definition}
\newtheorem{definition}[theorem]{Definition}
\newtheorem{remark}[theorem]{Remark}

\newcommand\R{\mathbb R}
\newcommand\M{\mathbb M}
\newcommand\N{\mathbb N}
\newcommand\Oo{\mathcal{O}}
\newcommand\Z{\mathbb{Z}}

\newcommand\iO{\int_\Omega}
\newcommand{\iOQ}{\int_{\Omega}\int_Q}
\newcommand{\iQQ}{\int_Q\int_Q}
\newcommand\iQ{\int_Q}
\newcommand\e{\varepsilon}

\newcommand\mn{\mathbb{M}^{N\times N}}
\newcommand\PP{\mathbb{P}}
\newcommand\QQ{\mathbb{Q}}
\newcommand\rd{\R^d}
\newcommand\Rm{\R^m}
\newcommand\rl{\R^l} 
\newcommand\rln{\R^{l\times N}}
\newcommand\en{\e_n}
\newcommand\ek{\e_k}
\newcommand\limn{\lim_{n\to +\infty}}
\newcommand\liminfn{\liminf_{n\to +\infty}}
\newcommand\ep{\varepsilon}
\newcommand{\scal}[2]{\langle #1,\,#2 \rangle}

\newcommand{\dist}{\mathrm{dist}}
\newcommand\wk{\to}
\newcommand\wks{\to}
\newcommand\wkts{\overset{2-s}{\rightharpoonup}}
\newcommand\sts{\overset{2-s}{\to}}

\newcommand{\C}{\mathbb{C}}
\newcommand{\ck}{\mathcal{C}^{\pdeor_k}}
\newcommand{\ckn}{\mathcal{C}^{\pdeor_{k_n}}}
\newcommand{\cc}{\mathcal{C}}
\newcommand{\F}{\mathscr{F}_{\pdeor}}
\newcommand{\fk}{\mathscr{F}_{\pdeor_k}}
\newcommand\Rn{\mathbb{R}^N}
\newcommand{\FF}{\mathcal{F}_{\pdeor}}
\newcommand{\FB}{\overline{\mathcal{F}}_{\pdeor}}
\newcommand{\FFN}{\mathcal{F}_{\pdeor(n\cdot)}}
\newcommand{\FBN}{\overline{\mathcal{F}}_{\pdeor(n\cdot)}}
\newcommand\rn[1]{R_{#1}^k}
\newcommand{\rnn}{R^k}
\newcommand\sn{S^k}
\newcommand\pdeep[1]{{\mathscr{A}}^{\rm div}_{#1}}
\newcommand\pdeepk[1]{{\mathscr{A}}_{k, #1}^{\rm div}}
\newcommand\pdeepn[1]{{\mathscr{A}_n}^{\rm div}_{#1}}
\newcommand\pde{\mathfrak{A}}
\newcommand\pdeor{\mathscr{A}}
\newcommand{\px}{\mathbb{T}_x}
\newcommand{\py}{\mathbb{T}_y}
\newcommand{\A}{\mathbb{A}}
\newcommand\aaa{{\mathscr{A}}}
\newcommand{\cx}[0]{\mathcal{C}_{x}}
\newcommand{\qa}[0]{Q_{\pdeor}}
\newcommand{\iq}[0]{\int_Q}
\newcommand{\qan}[0]{\qa^n}
\newcommand{\zn}{{Z}_{\e}}
\newcommand{\qn}{Q_{\e,z}}
\newcommand{\ze}{{Z}_{\e}}
\newcommand{\qe}{Q_{\e,z}}
\newcommand{\note}[1]{{\color{red}NOTE: #1}}
\newcommand{\HH}{\mathcal{H}}
\newcommand{\intt}{\int_0^{T}}
\newcommand{\intom}{\int_{\Omega}}
\newcommand{\mthree}{\R^{d\times d}}
\newcommand{\mb}{\mathcal{M}_b(\Omega\cup\Gamma_{\text{\sc d}};\mthree_{\rm dev})}
\newcommand{\V}{\mathcal{V}}

\newcommand\be[1]{\begin{equation}\label{#1}}
\newcommand\ee{\end{equation}}
\newcommand\ba[1]{\begin{align}\label{#1}}
\newcommand\ea{\end{align}}
\newcommand\bas{\begin{align*}}
\newcommand\eas{\end{align*}}
\newcommand\nn{\nonumber}
\newcommand\eep{\exp\Big(-\frac{t}{\ep}\Big)}
\newcommand\U{\mathscr U}
\newcommand\rti{\eta_{\tau,i}}
\newcommand\rtit{\eta_{\tau,i+2}}
\newcommand\rtio{\eta_{\tau,i+1}}
\newcommand\K{\mathbb{K}}
\newcommand\oo[1]{\rm{O}(#1)}
\newcommand\DT[1]{\mathchoice
                 {{\buildrel{\hspace*{.1em}\text{\LARGE.}}\over{#1}}}
                 {{\buildrel{\hspace*{.1em}\text{\Large.}}\over{#1}}}
                 {{\buildrel{\hspace*{.1em}\text{\large.}}\over{#1}}}
                 {{\buildrel{\hspace*{.1em}\text{\large.}}\over{#1}}}}
\newcommand\DDT[1]{\mathchoice
   {{\buildrel{\hspace*{.13em}\text{\LARGE.\hspace*{-.1em}.}}\over{#1}}}
   {{\buildrel{\hspace*{.1em}\text{\Large.\hspace*{-.1em}.}}\over{#1}}}
   {{\buildrel{\hspace*{.1em}\text{\large.\hspace*{-.1em}.}}\over{#1}}}
   {{\buildrel{\hspace*{.1em}\text{\large.\hspace*{-.1em}.}}\over{#1}}}}
\newcommand\DDDT[1]{\mathchoice
   {{\buildrel{\hspace*{.13em}\text{\LARGE.\hspace*{-.13em}.\hspace*{-.13em}.}}\over{#1}}}
   {{\buildrel{\hspace*{.1em}\text{\Large.\hspace*{-.1em}.\hspace*{-.1em}.}}\over{#1}}}
   {{\buildrel{\hspace*{.1em}\text{\large.\hspace*{-.1em}.\hspace*{-.1em}.}}\over{#1}}}
   {{\buildrel{\hspace*{.1em}\text{\large.\hspace*{-.1em}.\hspace*{-.1em}.}}\over{#1}}}}
\renewcommand\d{\mathrm{d}}
\newcommand\dx{\mathrm{d}x}
\newcommand\dt{\mathrm{d}t}
\newcommand\ds{\mathrm{d}s}
\newcommand{\lineunder}[2]{\LU{\begin{array}[t]{c}\underbrace{#1}\vspace*{.5em}\end{array}}{\mbox{\footnotesize\rm #2}}}
\newcommand{\linesunder}[3]{\LSU{\begin{array}[t]{c}\underbrace{#1}\vspace*{.5em}\end{array}}{\mbox{\footnotesize\rm #2}}{\mbox{\footnotesize\rm#3}}}
\newcommand{\LU}[2]{\begin{array}[t]{c}#1\vspace*{-1em}\\_{#2}\end{array}}
\newcommand{\LSU}[3]{\begin{array}[t]{c}#1\vspace*{-1em}\\_{#2}\vspace*{-.5em}\\_{#3}\end{array}}
\newcommand\tr{{\rm tr}\,}
\newcommand\MM{\mathbb{M}}
\newcommand{\D}{\mathbb{D}}
%
%
\newcommand\SYLD{\sigma_{_{\rm YLD}}}
\def\bbC{\mathbb{C}}
\def\bbD{\mathbb{D}}
\def\Colon{\!\colon\!}
\def\Cdot{\!\cdot\!}
\def\In{\!\in\!}
\def\SEFS{s_{_{\rm DAM}}}
\def\Rdev{\R_{\rm dev}^{d\times d}}
\def\Meas{{\rm Meas}}
\def\taur{\chi}
\def\pl{\partial}

\bigskip\bigskip\bigskip

\centerline{\LARGE\bf Dynamic perfect plasticity and damage in viscoelastic
solids}

\bigskip\bigskip

\centerline{{\sc Elisa Davoli}\footnote{Faculty of Mathematics, University of Vienna, 
Oskar-Morgenstern-Platz 1, A-1090 Vienna, Austria},\ \ 
{\sc Tom\'a\v s Roub\'\i\v cek}
\footnote{Mathematical Institute, Charles University,
Sokolovsk\'a 83, CZ-186~75~Praha~8,  Czech Republic}
\footnote{Institute of Thermomechanics, Czech Academy of Sciences,
Dolej\v skova 5, CZ-182~00~Praha~8, Czech Republic},\ \ 
{\sc Ulisse Stefanelli}$^{1,}$\footnote{Istituto di Matematica Applicata e Tecnologie
  Informatiche {\it E. Magenes}, Consiglio Nazionale delle Ricerche,
  via Ferrata 1, I-27100 Pavia, Italy} 
}


\bigskip

\begin{abstract}
{\bf Abstract}.  In this paper we analyze an isothermal and isotropic model 
for  viscoelastic media combining linearized perfect plasticity (allowing for 
concentration of plastic strain and development of shear bands) and damage 
effects in a dynamic setting. The interplay between the  viscoelastic rheology
with inertia, elasto-plasticity, and unidirectional rate-dependent 
incomplete damage affecting both the elastic and viscous response, as well as 
the plastic yield stress, is rigorously characterized by showing
existence of weak solutions to the constitutive  and balance 
equations of the model. The analysis relies on the notions of plastic-strain 
measures and bounded-deformation displacements, on sophisticated 
time-regularity estimates to establish a duality between acceleration and 
velocity of the elastic displacement, on the theory of rate-independent 
processes for the energy conservation in the dynamical-plastic part, and on 
the proof of the strong convergence of the elastic strains. Existence of a 
suitably defined weak solutions is proved rather constructively by using a 
staggered two-step time discretization scheme.

\medskip

\noindent{\bf Keywords:} Perfect plasticity, inertia, cohesive damage, 
Kelvin-Voigt viscoelastic rheology,
functions of bounded deformation,  staggered time discretisation, 
weak solution.
\medskip

\noindent{\bf AMS Subj. Classificaiton:}
35Q74, 
37N15, 
74C05, 
74R05. 
\end{abstract}


\section{Introduction}
\GGG 
Plasticity and damage are inelastic phenomena providing the
macroscopical evidence of defect formation \UUU and evolution at the
\GGG atomistic level. Plasticity results from the accumulation of slip defects (dislocations), which determine the behavior of a body to change from elastic and reversible to plastic and irreversible, once the magnitude of the stress reaches a certain threshold and a plastic flow develops.
Damage evolution originates from the formation of cracks and voids in the microstructure of the material. 

The mathematical modeling of inelastic phenomena is a very active
research area, at the triple point between mathematics, physics, and
materials science. \UUU A \GGG vast literature concerning damage in \UUU
viscoelastic \GGG materials, both in the quasistatic and the dynamical
setting \UUU is currently available. \GGG We refer, e.g., to \cite{MielkeThomas,MieRou15RIST,Roub??MDDP,LazzaroniRossiThomasToader,ZhangGross} and the references therein for an overview of the main results. 

The interplay between plasticity and damage has been \UUU already \GGG
extensively investigated, \UUU prominently in the quasistatic
framework. \GGG The interaction between damage and strain gradient
plasticity \UUU is addressed \GGG  in \cite{Crismale2017} \UUU whereas
a perfect-plastic model \GGG  has been proposed in
\cite{AlessiMarigoVidoli}, where the one-dimensional response is also
studied. Existence results in general dimensions have been obtained in
\cite{Crismale2015, CrismaleLazzaroni}, see also
\cite{CrismaleOrlando} for some recent associated lower semicontinuity
results. The coupling between damage and rate-independent small-strain
plasticity with hardening is the subject of
\cite{BonettiRoccaRossiThomas, ThomasRossi,
  RoubValdman2016}. Quasistatic perfect plasticity and damage with
healing are analyzed in \cite{RoubValdman2017}. The identification of
fracture models as limits of damage coupled with plasticity has \UUU
also been considered \GGG \cite{DalmasoIurlano, DalmasoOrlandoToader}.


The analysis of dynamic perfect plasticity without damage has been initiated in
\cite{anzellotti.luckhaus}. A derivation of the equations via
vanishing hardening, and vanishing \UUU viscoplasticity \GGG has been
performed in \cite{Chel01PPZR,chelminski}. A generalization via \UUU
the so-called \GGG  cap-model approximation has been obtained in \cite{babadjian.mora}. An approximation of the equations of dynamic plasticity relying on the minimization of a parameter-dependent functional defined on trajectories is the subject of \cite{DavSte??DPPC}, whereas an alternative approach based on hyperbolic conservation laws has been proposed in \cite{BabadjianMifsud}. Dimension reduction for dynamic perfectly plastic plates has been carried on in \cite{MagMor16DEMP}. Convergence of dynamic models to quasistatic ones has been analyzed in \cite{dalmaso.scala,Ross??VPPT}.\\

To our best knowledge, the combination of perfect plasticity, damage,
and inertia has been so far tackled in the engineering and geophysical
literature (see, e.g., \cite{XuEtAl,DuanDay,GatuingtPiajudier}),
whilst a mathematical counterpart to the applicative analysis is still
missing. The focus of this paper is to provide a rigorous analysis of
an \EEE isothermal and isotropic \GGG model  for \UUU viscoelastic
\GGG media combining both
small-strain perfect plasticity and damage effects in a dynamic
setting.

\UUU More specifically, \GGG our main result (Theorem \ref{thm:main})
\UUU shows \GGG existence of \UUU suitably weak \GGG solutions to the following system of equations and differential inclusions, complemented by suitable boundary conditions and initial data
\begin{subequations}\begin{align}\label{system-u-int}
&\rho \DDT u-{\rm div}\,\sigma
=f,\ \ \ 
\sigma:=\C(\alpha)e_{\rm el}+\D(\alpha)\DT e_{\rm el},\ \ e_{\rm el}=e(u)-\pi,\\
&\label{system-p-int}
\SYLD(\alpha){\rm Dir}(\DT{\pi})\ni {\rm dev}\,\sigma,
\\
&\label{system-alpha-int}
\partial \zeta(\DT\alpha)
+\frac12 \C'(\alpha)e_{\rm el}:e_{\rm el}\ni \phi'(\alpha) +
{\rm div}\, (\kappa |\nabla  \alpha|^{p-2}\nabla  \alpha) ,
\end{align}
\end{subequations}
where $u,\, \pi$, and $\alpha$ denote the displacement, the plastic
strain, and the damage variable, respectively, $\C(\cdot)$,
$\D(\cdot)$, and $\SYLD(\cdot)$ are the damage-dependent elasticity
tensor, viscosity tensor, and yield surface, \UUU and \GGG
$e(u)=(\nabla u+\nabla u^{\top})$ \UUU is \GGG  the linearized strain. The notation {\rm Dir} stands for the set-valued ``direction'' (see Subsection \ref{subs:kin}), {\rm dev} $\sigma$ identifies the deviatoric part of the stress $\sigma$, namely ${\rm dev}\,\sigma:=\sigma- {\rm tr }\,(\sigma){\rm Id}/d$, $\zeta$ is the local potential 
of dissipative damage-driving force (see \eqref{def-of-zeta}), \UUU
constraining the damage process to be unidirectional (no
healing). \GGG Finally $\phi$ is the energy associated to the creation
of microvoids or microcracks during the damaging process, $\kappa$ is
the length scale \UUU of \GGG  the damage profile, and $\rho$ the mass
density. We refer to Section \ref{sec:notation} for the precise
setting of the problem, \UUU the \GGG definition of weak solution to
\eqref{system-u-int}--\eqref{system-alpha-int}, and \UUU the \GGG statement of Theorem \ref{thm:main}. 
\EEE

\UUU The analysis of model \eqref{system-u-int} \GGG presents several technical challenges. \EEE Perfect plasticity 
allows for plastic strain concentrations
along the (possibly infinitesimally thin) slip-bands and  
\GGG calls for weak formulations in the spaces of bounded Radon \EEE
measures \UUU for plastic strains \EEE and bounded-deformation
($BD$) \UUU for displacements. This \EEE requires a delicate notion of stress-strain duality (see Subsection \ref{subs:weak}). Considering inertia and the related kinetic energy \GGG renders the analysis quite delicate because of \EEE the interaction of possible 
elastic waves with nonlinearly responding slip bands, 
as pointed out already in \cite{BarRou13NATC}. \GGG Various natural extensions such as \EEE allowing healing 
instead of unidirectional damage, or mutually independent damage in the viscous and the elastic
response (in contrast to \eq{eq:def-D} below), or \UUU different
damage behaviors in relation \EEE to compression/tension mode leading to a non-quadratic stored 
energy, or an enhancement by heat generation/transfer 
with some thermal coupling to the mechanical part, seem difficult and remain currently
open. 

\GGG
The proof strategy relies on a staggered discretization scheme, in which at each time-step we first identify the damage variable as a solution to the damage evolution equation, and we then determine the plastic strain and elastic displacements as minimizer of a damage-dependent energy inequality (see Section \ref{sec:time-discrete}). A standard test of \eqref{system-u-int}--\eqref{system-alpha-int} leads to the proof of a first a-priori estimate in Proposition \ref{prop:unif-estimates}. In order to ensure the strong convergence of the time-discrete elastic strains $e_{\rm el}$, \EEE needed 
for the limit passage in the damage flow rule, \GGG a further higher order test is performed in Proposition \ref{prop:higher-order}. The convergence of the elastic strains is then achieved by means of a delicate limsup estimate (see Proposition \ref{prop:strong-convergence}). \EEE Due to the failure of energy conservation under basic data qualification, \GGG the flow rule is only recovered, in the limit, in the form of an energy inequality (see Remark \ref{rk:no-en-cons}). 

A motivation for tackling the simultaneous occurrence of dynamical perfect plasticity and damaging is the mathematical modeling of cataclasite zones in geophysics. During fast slips, lithospheric faults in elastic rocks tend to emit elastic (seismic) waves, which in turn determine the occurrence of (tectonic) earthquakes, and the local arising of cataclasis. This latter phenomenon consists in a gradual fracturing of mineral grains into core zones of lithospheric faults, which tend to arrange themselves into slip bands, sliding plastically on each other without further fracturing of the material. On the one hand, cataclasite core zone are often very narrow (sometimes
centimeters wide) in comparison
with the surrounding compact rocks (which typically extend for many kilometers),
and can be hence modeled for 
rather small time scales (minutes of ongoing earthquakes or years between 
them, rather than millions of years) via small-strain perfect (no-gradient) plasticity. On the other hand
the partially damaged area surrounding the thin cataclasite
core can by relatively wide, and thus calls for a modeling via
gradient-damage theories (see \cite{Roub17GMHF,tr-us-seismo}).

The novelty of our contribution is threefold. First, we extend the
mathematical modeling of damage-evolution effects to an inelastic
setting. Second, we characterize the interaction between damage onset
and plastic slips formation in the framework of perfect plasticity,
with no gradient regularization and in the absence of
hardening. Third, we complement the study of dynamic perfect
plasticity, by keeping track of the effects of damage both on the
plastic yield surface, and on the \UUU viscoelastic \EEE behavior of the material.

The paper is organized as follows: In Section~\ref{sec:notation}, we introduce some basic notation and modeling assumptions, and we state our main existence result. Section~\ref{sec:formal} highlights the formal strategy that will be employed afterward for the proof of Theorem \ref{thm:main}, whereas Section~\ref{sec:time-discrete} focuses on the formulation of our staggered two-step discretization scheme. In Section ~\ref{sec-apriori} we establish some a-priori energy estimates. Finally Section~\ref{sec:proof-main} is devoted to the proof of the main result.
\EEE

\section{Setting of the problem and statement of the main result}
\label{sec:notation}
We devote this section to specify the mathematical setting of the model, and 
to present our main result. We first introduce some basic notation and 
assumptions, and we recall some notions from measure theory. 

In what follows, let $\Omega\subset \R^d$, $d\in \{2,3\}$ be a bounded open set with $C^2$ boundary. 
In our model, the domain $\Omega$ represents the reference configuration of a 
linearly \UUU viscoelastic, \EEE perfectly plastic Kelvin-Voigt  
body subject to a possible damage in its elastic as well as in its viscous and
plastic response.

We assume that the boundary $\partial\Omega=:\Gamma$ is 
partitioned into the union of two disjoint sets $\Gamma_{\text{\sc d}}$ and $\Gamma_{\text{\sc n}}$. In particular, we require $\Gamma_{\text{\sc d}}$ to be a connected open subset of $\Gamma$ (in the relative topology of $\Gamma$) such that $\partial_{\Gamma}\Gamma_{\text{\sc d}}$ is a connected, \GGG $(d-2)$-dimensional, \EEE $C^2$ manifold, whereas $\Gamma_{\text{\sc n}}$ is defined as $\Gamma_{\text{\sc n}}:=\Gamma\setminus \Gamma_{\text{\sc d}}$.

For any map $f:[0,T]\times\R^d\to\R$ we will denote
by $\dot{f}$ its time derivative, and by $\nabla f$ its spatial
gradient. We will adopt the notation $\mthree$ to indicate the set of 
$d\times d$ 
matrices. Given $M,N\in\mthree$, their scalar
product will be denoted by $M:N:= {\rm tr} (M^\top N)$ where ${\rm tr}$ is the
  trace operator, and the superscript stands for transposition. We will write ${\rm dev }\,M$ to identify the
  deviatoric part of $M$, namely ${\rm dev}\,M:=M- {\rm tr }\,(M){\rm
    Id}/d$, where ${\rm Id}$ is the identity matrix. The symbols $\mthree_{\rm sym}$ and $\mthree_{\rm dev}$ will represent the set of symmetric $d\times d$ matrices, and that of symmetric matrices having null trace, respectively.

\subsection{Function spaces, measures and functions with bounded deformation}
 
We use the standard notation $L^p$, $W^{k,p}$, and $L^p(0,T;X)$ 
or $W^{1,p}(0,T;X)$ for Lebesgue, Sobolev, and Bochner or Bochner-Sobolev
spaces. By $C_{\rm w}(0,T;X)$ we denote the space of weakly continuous
mappings \GGG with value \EEE in the Banach space $X$.
\GGG We also \EEE use the shorthand convention $H^k:=W^{k,2}$.  

Given a Borel set $B\subset\R^d$ the symbol
$\mathcal{M}_b(B;\R^m)$ denotes the space of bounded
Borel measures on $B$ with values in $\R^m\ (m \in \N)$.  When
$m=1$ we will simply write $\mathcal{M}_b(B)$. We will endow
$\mathcal{M}_b(B; \R^m)$ with the norm $\|\mu\|_{\mathcal{M}_b(B;\R^m)}:=|\mu|(B)$, where
$|\mu|\in\mathcal{M}_b(B)$ is the total variation of the measure $\mu$. 

If the relative topology of $B$ is locally compact, by the Riesz
representation Theorem the space $\mathcal{M}_b(B; \R^m)$ can be identified 
with the dual of $C_0(B;\R^m)$, which is the space of continuous functions 
$\varphi:B\to \R^m$ such that
the set $\{|\varphi|\geq\delta\}$ is compact for every $\delta>0$. The weak* 
topology on $\mathcal{M}_b(B;\R^m)$ is defined using this duality.

The space $BD(\Omega;\R^d)$ of functions with {\em bounded deformation} is
the space of all functions $u\in L^1(\Omega;\R^d)$ whose symmetric gradient  
$$
e(u):=\frac{\nabla u+(\nabla u)^{\top}}{2}
$$
(defined in the sense of distributions) 
belongs to $\mathcal{M}_b(\Omega;\mthree_{\rm sym})$. It is easy to see that $BD(\Omega;\R^d)$ is a 
Banach space \UUU when \EEE endowed with the norm
$$
\|u\|_{L^1(\Omega;\R^d)} +\|e(u)\|_{\mathcal{M}_b(\Omega;\mthree_{\rm sym})}.
$$
A sequence $\{u^k\}$ is said to converge to $u$ weakly* in 
$BD(\Omega;\R^d)$ if $u^k\wk u$ weakly in
$L^1(\Omega;\R^d)$ and $e(u^k)\wk e(u)$ weakly* in 
$\mathcal{M}_b(\Omega;\mthree_{\rm sym})$. Every bounded
sequence in $BD(\Omega;\R^d)$ has a weakly* converging subsequence. In our setting, since $\Omega$ is bounded and has $C^2$ boundary, $BD(\Omega;\R^d)$ can be embedded into $L^{d/{(d-1)}}(\Omega;\R^d)$ and every function $u\in BD(\Omega;\R^d)$ has a trace,
still denoted by $u$, which belongs to $L^1(\Gamma;\R^d)$. For every
nonempty subset \GGG $\gamma$ \EEE of \UUU $\Gamma_{\text{\sc d}}$  \EEE which
is open in the relative topology of \GGG $\Gamma_{\text{\sc d}}$ \EEE, there exists a constant
$C>0$, depending on $\Omega$ and \GGG $\gamma$, \EEE such that the following Korn inequality holds true
\begin{equation}\label{eq:korn}
\|u\|_{L^1(\Omega;\R^d)}\leq C\|u\|_{L^1(\GGG \gamma
  \EEE ;\R^d)}+C\|e(u)\|_{\mathcal{M}_b(U;\mthree_{\rm sym})}
\end{equation}
(see \cite[Chapter~II, Proposition~2.4 and Remark~2.5]{temam}). 
For the general properties of the space $BD(\Omega;\R^d)$ we refer to~\cite{temam}.

\subsection{State of the system and admissible displacements and strains}
At each time $t\in [0,T]$, the \UUU viscoelastic \EEE perfectly-plastic
behavior of the body is \UUU described by \EEE three basic state variables: the displacement $u(t):\Omega\to \R^d$, the plastic strain 
$\pi(t):\Omega\to \mthree_{\rm dev}$, and the damage variable $\alpha(t):\Omega\to [0,1]$.  
In particular, we adopt the convention (used in mathematics, in contrast to the opposite 
convention used in engineering and geophysics) that $\alpha=1$ corresponds to the 
\UUU undamaged \EEE elastic material, whereas $\alpha=0$ describes the situation in which the material 
is totally damaged. The abstract state $q$ will be here \GGG given by \EEE the triple $q=(u,\pi,\alpha)$.

On $\Gamma_{\text{\sc d}}$ 
we prescribe a boundary datum $u_{\text{\sc d}}
\in  H^{1/2}(\Gamma_{\text{\sc d}};\R^d)$, later being \UUU considered
to be \EEE time dependent. 
With a slight abuse of notation we also denote by 
$u_{\text{\sc d}}$ 
a $H^1(\Omega;\R^d)$-extension of the boundary condition to the set $\Omega$. 

The {\em set of admissible displacements and strains} for the boundary datum 
$u_{\text{\sc d}}$ is given by  
\begin{align}
\nn\pdeor(u_{\text{\sc d}}
):=&\Big\{(u,e_{\rm el},\pi)\in \big(BD(\Omega;\R^d)\cap L^2(\Omega)\big)\times L^2(\Omega;\mthree_{\rm sym})\times \mb:\,\\
&\quad e(u)=e_{\rm el}+\pi\text{ in }\Omega,
\label{eq:def-A-w}\quad \pi=(u_{\text{\sc d}}
-u)\odot\nu_\Gamma^{}\mathcal{H}^{d-1}\text{ on }\Gamma_{\text{\sc d}}\Big\},
\end{align}
where $\odot$ stands for the symmetrized tensor product, namely 
$$
a\odot b:=(a\otimes b+b\otimes a)/2\quad\forall\,a,b\in\R^d,
$$
$\nu_\Gamma^{}$ is the outer unit normal to $\Gamma$, and 
$\HH^{d-1}$ is the $(d-1)$-dimensional Hausdorff measure. 
\UUU Note that the kinematic relation $e(u) = e_{\rm el}+\pi$ in
$\nn\pdeor(u_{\text{\sc d}})$  is classic in linearized elastic
theories and it is usually
referred to as additive strain decomposition. \EEE

We point out that the constraint
\be{eq:relaxed-bc}
\pi=(u_{\text{\sc d}}-u)\odot\nu_\Gamma^{}\mathcal{H}^{d-1}\text{ on }\Gamma_{\text{\sc d}}
\ee
is a relaxed formulation of the boundary condition $u=u_{\text{\sc d}}$ on $\Gamma_{\text{\sc d}}$;
see also \cite{mora}. As remarked in \cite{dalmaso.desimone.mora06}, the mechanical meaning of 
\eqref{eq:relaxed-bc} is that whenever the boundary datum is not attained a plastic slip 
develops, whose amount is directly proportional to the difference between the displacement 
$u$ and the boundary condition $u_{\text{\sc d}}$.

\subsection{Stored energy}\label{subs:el-visc}

Let $\mathcal{L}_{\rm sym}(\mthree_{\rm sym})$ denote the space of
linear symmetric \GGG (self-adjoint) \EEE operators $\mthree_{\rm sym}\to\mthree_{\rm sym}$, 
being understood as 4th-order 
symmetric tensors. 

We assume the elastic tensor $\C:\R\to \mathcal{L}_{\rm sym}(\mthree_{\rm sym})$ to be 
continuously differentiable, and nondecreasing in the sense of the L\"owner 
ordering, i.e.\ the ordering of $\mthree_{\rm sym}$ \GGG with respect to \EEE the cone of \GGG positive \EEE
semidefinite matrices. Additionally, we require $\C(\alpha)$ to be 
positive semi-definite for every $\alpha\in\R$. Note that, in view of the 
pointwise semi-definiteness of $\C$, the possibility of having complete 
damage in the elastic part is also encoded in the model. We additionally 
assume that $\C(\alpha)=\C(0)$ for every $\alpha<0$, and that
$\C'(0)=0$. \UUU This corresponds to \EEE the situation in which the damage is cohesive. 
%

The \emph{stored energy} of the model will be given by
\begin{align}\label{def-of-E}
\mathscr{E}(q)=\mathscr{E}(u,\pi,\alpha)
=\int_{\Omega}\GGG \Big(\EEE\frac12\C(\alpha)e_{\rm el}:e_{\rm el}-\phi(\alpha)
+\frac\kappa p|\nabla\alpha|^p\GGG\Big)\EEE\,\d x
\ \ \ \ \text{ with }\ \ e_{\rm el}=e(u)-\pi,
\end{align}
\GGG where $\phi:\R\to \R$ stands \EEE for the specific \emph{energy of damage}, 
motivated by extra energy of 
microvoids or microcracks created by degradation of the material during \GGG the \EEE damaging process\GGG, whereas $\kappa$ represents a length scale for the damage profile. \EEE
When $\phi'(\alpha)>0$, the damage evolution is an activated processes,
even if there is no activation threshold in the dissipation potential, as 
indeed considered in \eq{def-of-zeta} below.

\UUU For the sake of allowing full generality to the choice of initial
conditions, we will assume that  ${\rm dev}\bbC e=\bbC{\rm
  dev}\,e$. Note that this is the case for 
  isotropic materials. \EEE

\subsection{Other \GGG ingredients:  dissipation and kinetic \GGG energy \EEE}\label{sub:diss}
\UUU For the sake of \EEE notational simplicity, we consider isotropic materials as far as 
plastification \GGG is concerned. \EEE

Let \UUU the {\it yield stress} $\SYLD$ as a function of damage \EEE $\SYLD:[0,1]\to (0,+\infty)$ be continuously differentiable and 
non-decreasing. 
For every 
$\pi\in\mathcal{M}_b(\Omega\cup\Gamma_{\text{\sc d}};\mthree_{\rm dev})$ let 
$\d\pi/\d|\pi|$ be the Radon-Nikod\'ym derivative of $\pi$ with respect to
its \GGG total \EEE variation $|\pi|$. Assuming that $\alpha:[0,T]\times \Omega\to [0,1]$ is 
continuous, we consider the positively one-homogeneous function 
$
 M\mapsto\SYLD(\alpha)|M|$ for every $M\in\mthree$, and, according to the 
theory of convex functions of measures \cite{goffman.serrin}, we introduce the 
\GGG functional \EEE 
$$
\mathcal{R}(\alpha,\pi):=
\int_{\Omega\GGG \cup \Gamma_{\text{\sc d}}}\SYLD(\alpha)\frac{\d \pi}{\d|\pi|}\,\d |\pi|
$$
for every $\pi\in\mathcal{M}_b(\Omega\cup\Gamma_{\text{\sc d}};\mthree_{\rm dev})$.

In what follows, we will refer to $\mathcal{R}$ as to the 
{\it damage-dependent plastic dissipation potential}. Note that, by 
Reshetnyak's lower semicontinuity theorem (see \cite[Theorem 2.38 ]{ambrosio.fusco.pallara}), 
the functional $\mathcal{R}$ is lower-semicontinuous in its second variable 
with respect to the weak* convergence in \mbox{$\mb$}.

For $\alpha$ continuous and such that $\DT{\alpha}\leq 0$ in 
$[0,T]\times\Omega$, we define the 
{\em $\alpha $-weighted $\mathcal{R}$-dissipation} of a map $t\mapsto \pi(t)$ 
in the interval $[s_1,s_2]$ as
\begin{align}
D_{\mathcal{R}}(\alpha;\pi;s_1,s_2):=\sup &\Bigg\{\sum_{j=1}^n
\mathcal{R}\big(\alpha(t_j),\pi(t_{j}){-}\pi(t_{j-1})\big): \
s_1\le t_0< t_1<\dots< t_n\le s_2,\ n\in\N\Bigg\}.
\label{eq:def-a-R-diss}
\end{align}

We will work under the assumption that the damage is unidirectional, i.e.\
$\DT{\alpha}\leq 0$.
Constraining the rate rather than the state itself,  
this constraint is to be incorporated into the dissipation potential.
For a (small) damage-viscosity parameter \GGG $\eta>0$ \EEE, we define the local potential 
of dissipative damage-driving force as
\begin{align}\label{def-of-zeta}
\zeta(\DT\alpha):=\begin{cases}\displaystyle\frac12\eta\DT\alpha^2
&\text{if }\DT\alpha\le0,
\\+\infty.&\text{otherwise} \end{cases}
\end{align}

\UUU Let the viscoelastic \EEE tensor $\D:\R\to \mathcal{L}_{\rm
  sym}(\mthree_{\rm sym})$ \UUU be given and \EEE define the overall potential of dissipative forces 
\begin{align}\nonumber
\mathscr{R}(q;\DT q)&=\mathscr{R}(\alpha;\DT u,\DT\pi,\DT\alpha)
\\&=\int_\Omega\GGG\Big(\EEE\frac12\D(\alpha)\DT e_{\rm el}:\DT e_{\rm el}
+\zeta(\DT\alpha)\GGG\Big)\EEE\,\d x+\int_{\UUU \Omega \cup \Gamma_{\text{\sc d}}}\SYLD(\alpha)\frac{\d\DT\pi}{\d|\DT\pi|}\,\d |\DT\pi|
\ \ \ \text{ where }\ \DT e_{\rm el}=e(\DT u)-\DT\pi.
\label{def-of-R}\end{align}

\GGG Let $\rho\in L^{\infty}(\Omega)$, with $\rho>0$ almost everywhere in $\Omega$ denote the mass density. We will \EEE
additionally consider the  \emph{kinetic energy} given by \EEE
\begin{equation}
\mathscr{T}(\DT u)=\int_{\Omega}\frac{1}{2}\rho|\DT{u}|^2\,\d x.\label{ottostar}
\end{equation}

\subsection{Governing equations by Hamilton variational principle}
\label{subs:kin}

\UUU We formulate the model via \EEE {\it Hamilton's 
variational principle}\index{Hamilton variational principle} generalized for 
dissipative systems \cite{Bedf85HPCM}. \UUU This prescribes \EEE that, among all admissible motions $q=q(t)$ 
on a fixed time interval $[0,T]$ \UUU given the initial and final
states $q(0)$ and $q(T)$, \EEE the actual motion is \UUU a stationary
point of the {\it action} \EEE
\begin{align}\label{hamilton}
\int_0^T{\mathscr L}\big(t,q,\DT{q}\big)\,\d t
\end{align}
where $\DT{q}=\frac{\pl}{\pl t}{q}$ and the 
\emph{Lagrangean}\index{Lagrangean} ${\mathscr L}(t,q,\DT{q})$ is
defined \UUU as \EEE 
\begin{align}\label{Lagrangean}
&{\mathscr L}\big(t,q,\DT q\big):={\mathscr T}\big(\DT q\big)
-{\mathscr E}(q)+\langle F(t),q\rangle\ \ \ \text{ \UUU with \EEE }\ \ F=F_0(t)-\pl_{\DT q}{\mathscr R}(q,\DT q)\,.
\end{align}
\UUU This corresponds to \EEE 
the sum of \UUU external \EEE time-dependent loading and the (negative) nonconservative 
force assumed for a moment fixed. 
In addition to ${\mathscr E}$, ${\mathscr R}$, and  ${\mathscr T}$
from Sections~\ref{subs:el-visc} and \ref{sub:diss},
we define the outer loading $F_0$ as $\langle F_0(t),q\rangle=
\int_\Omega f\cdot u\,\d x$, \GGG where $f$ is a time-dependent external body load. \EEE

\UUU The corresponding Euler-Lagrange equations read \EEE
\begin{align}
\pl_u^{}{\mathscr L}\big(t,q,\DT{q}\big)
-\frac{\d}{\d t}\pl_{\DT q}^{}{\mathscr L}\big(t,q,\DT{q}\big)=0.
\end{align}
This gives the 
abstract 2nd-order evolution equation 
\begin{align}\label{abstract-momentum-eq}
\UUU \partial^2{\mathscr T} \EEE\DDT q+\pl_{\DT q}{\mathscr R}(q,\DT q)+{\mathscr E}'(q)=F_0(t)
\end{align} 
where \UUU$ \pl$ \EEE indicates the (partial) G\^ateaux
differential. \UUU Let us now rewrite the abstract relation
\eqref{abstract-momentum-eq} in terms of our specific choices 
\eqref{def-of-E}, \eqref{def-of-zeta}-\eqref{ottostar}. We have \EEE

{the following equation/inclusion on $[0,T]\times\Omega$:}
 \begin{subequations}\label{system}
\begin{align}\label{system-u-new}
&\rho \DDT u-{\rm div}\,\sigma
=f,\ \ \ 
\sigma:=\C(\alpha)e_{\rm el}+\D(\alpha)\DT e_{\rm el},\ \ e_{\rm el}=e(u)-\pi,\\
&\label{system-p}
\SYLD(\alpha){\rm Dir}(\DT{\pi})\ni {\rm dev}\,\sigma,
\\
&\label{system-alpha}
\partial \zeta(\DT\alpha)
+\frac12 \C'(\alpha)e_{\rm el}:e_{\rm el}\ni \phi'(\alpha)\UUU +
{\rm div}\, (\kappa |\nabla  \alpha|^{p-2}\nabla  \alpha) \EEE ,
\end{align}
\end{subequations}
comple\GGG men\EEE ted 
by the boundary conditions 
\begin{align}\label{bc}
&
\sigma\nu_\Gamma^{}=0\quad \text{on }[0,T]\times\Gamma_{\text{\sc n}},\quad
u=u_{\text{\sc d}}\quad \text{on }[0,T]\times\Gamma_{\text{\sc d}},\quad
\nabla \alpha\cdot \nu_\Gamma^{}=0\quad\text{on }[0,T]\times\Gamma.
\end{align}
The notation ${\rm Dir}:\Rdev\rightrightarrows\Rdev$ in \eq{system-p}
means the set-valued ``direction'' mapping defined by
 ${\rm Dir}(\DT \pi):=[\pl|\cdot|](\DT\pi)$. \UUU In particular
$$ {\rm Dir} (\DT \pi)= 
\left\{
\begin{array}{ll}
 \DT \pi/|\DT \pi| \quad &\text{if} \ \  \DT\pi \not = 0\\[1mm]
\{d \in \Rdev \ : \ |d| \leq 1\}&\text{if} \ \  \DT\pi  = 0
\end{array}
\right.
$$
Relations \eqref{system-u-new}, \eqref{system-p}, and
\eqref{system-alpha} correspond to the equilibrium equation and constitutive relation, the plastic flow
rule, and the evolution law for damage, respectively.  
 
The above boundary-value problem is complemented with initial
conditions as follows \EEE , 
\begin{align}\label{initial conditions}
u(0)=u_0,\quad \pi(0)=\pi_0,\quad \alpha(0)=\alpha_0, \quad \DT{u}(0)=v_0.
\end{align}


We point out that the monotonicity of $\C$, combined with the 
unidirectionality \UUU ($\DT \alpha \leq 0$) \EEE of damage implies that 
\begin{equation}\label{eq:monoton1}
\DT \alpha \C'(\alpha)e:e\leq 0\quad\text{for every }e\in \mthree,
\end{equation}
namely $\DT \alpha \C'(\alpha)$ is negative semi-definite.  By the monotonicity 
of $\SYLD$, the unidirectionality of damage also \GGG yields \EEE that 
\begin{equation}\label{eq:monoton2}
\DT \alpha \SYLD'(\alpha)\leq 0.
\end{equation}

 
The \emph{energetics} of the model \eq{system}-\eq{bc}, obtained by standard tests of 
\eq{system}
\GGG successively \EEE against $\DT{u}$,  $\DT{\pi}$, and $\DT{\alpha}$,
is formally encoded by the following energy equality
\begin{align}\nonumber
&\linesunder{\int_{\Omega}\frac\rho2|\DT{u}(t)|^2\, \d x}{kinetic energy}{at time $t$}
+ 
\lineunder{\int_{\Omega}\frac12\C(\alpha(t)) e_{\rm el}(t):e_{\rm el}(t)-\phi(\alpha(t))
+\frac{\UUU \kappa}{p}|\nabla \alpha(t)|^p \,\d x}{stored energy at time $t$}
\\[-.7em]
&\nonumber\qquad\qquad\qquad 
+\lineunder{\int_0^{t}\!\!\int_{\Omega}\eta \DT{\alpha}^2
+
\D(\alpha) \DT{e}_{\rm el}:\DT{e}_{\rm el}\,\d x\,\d s
+
\SYLD(\alpha)|\DT{\pi}|\,\d x\,\d s}{dissipation on $[0,t]$}\\
&
\quad=\linesunder{\int_{\Omega}\frac{\rho}{2}|v_0|^2\,\d x}{
  \UUU kinetic energy}{\UUU at time 0}
\UUU
+ \lineunder{\int_{\Omega}\frac12\C(\alpha_0) e_{\rm el}(0):e_{\rm el}(0)-\phi(\alpha_0)
+\frac{\UUU \kappa}{p}|\nabla \alpha_0|^p \,\d x}{stored energy at
time $0$}
\EEE \nonumber\\
&\qquad\qquad\qquad 
+\linesunder{\int_0^t\!\!\int_{\Omega}  f\cdot \DT{u} \,\d x\,\d s }{energy
  of}{external bulk load}
\UUU +\linesunder{\int_0^t\!\!\int_{\Gamma_{\text{\sc d}}}  \sigma \nu_\Gamma
  \cdot \DT u_{\text{\sc d}} \,\d {\mathcal H}^{d-1}\,\d s }{energy
  of}{boundary condition}
\label{eq:formal-en-equality}
\end{align}\UUU
where the last term has to be interpreted in the sense of
\eqref{eq:def-bd-term} below. \EEE
A \UUU rigorous \EEE derivation of the energy equality above will be presented 
in Subsection \ref{subs:energetics}.

\subsection{Statement of the main result}
\UUU Let $p>d$ be given and  \EEE assume that the data of the problem satisfy the following conditions:
\begin{subequations}\label{eq:hp}\begin{align}
&u_0\in L^2(\Omega;\R^d)\cap
BD(\Omega;\R^d),\ \ \UUU v_0 \in H^1(\Omega;\R^d), \EEE \nonumber\\
& 
\pi_0\in \mb, \ \ \UUU \DT\pi_0 \in L^2(\Omega;\Rdev),\label{eq:hp-in-data}  \EEE
\\&\label{eq:hp-in-data+}\UUU
(u_0,e(u_0) - \pi_0, \pi_0) \in \mathscr{A}(u_{\text{\sc d}}(0)), \ \
(v_0,e(v_0) - \DT \pi_0, \DT \pi_0) \in \mathscr{A}(\DT u_{\text{\sc d}}(0)),\EEE
\\&\label{eq:hp-in-data++}
\alpha_0\in W^{1,p}(\Omega),\ \
\ 0\le\alpha_0\le1,\nonumber\\
& \UUU
\SYLD(\alpha_0){\rm Dir}(\DT \pi_0)\ni {\rm dev}\, ({\mathbb C}(\alpha_0)
(e(u_0){-}\pi_0) + {\mathbb D}(\alpha_0)(e(v_0){-}\DT \pi_0)), \EEE
\\
&\label{eq:hp-f}f\in {L^2}(0,T;L^2(\Omega;\R^d)),
\ \ \ u_{\text{\sc d}}\in W^{2,\infty}(0,T;L^2(\Omega;\R^d))\cap H^1(0,T;H^1(\Omega;\R^d)).
\end{align}\end{subequations}

\UUU The regularity requirements in \eqref{eq:hp} for $v_0$ and $\DT
\pi_0$ and the compatibility condition in \eqref{eq:hp-in-data++} are
needed in order to make some higher-order estimate rigorous, see
Subsection \ref{subs:higher-order}. \EEE

We \UUU now \EEE introduce \UUU the \EEE  notion of weak solution to \eqref{system}--\eqref{initial conditions}.

\begin{definition}[Weak solution to \eqref{system}--\eqref{initial conditions}]
\label{def:weak}
A quadruple
\begin{align*} 
 u&\in L^{\infty}(0,T;BD(\Omega;\R^d))\GGG \cap
                                H^2(0,T;L^2(\Omega;\R^d))\EEE
\\
e_{\rm el} &\in  H^1(0,T;L^2(\Omega;\mthree_{\rm sym})),\\
\pi &\in
  BV(0,T;\mb),\\
\alpha &\in \big(H^1(0,T;L^2(\Omega))\cap C_{\rm
    w}(0,T;W^{1,p}(\Omega)) 
\end{align*}
is a \emph{weak solution} to \eqref{system}--\eqref{initial conditions} if it satisfies \eqref{initial conditions}, 
and the following conditions are fulfilled:
\begin{itemize}
\item[(C1)] $(u(t),e_{\rm el}(t),\pi(t))\in \mathscr{A}(u_{\text{\sc d}}(t))$ for every $t\in [0,T]$ (see \eqref{eq:def-A-w});
\item[(C2)] The equilibrium equation \eqref{system-u-new} holds \UUU
  almost everywhere in $\Omega \times (0,T)$; \EEE  
\item[(C3)] 
The quadruple $(u,e_{\rm
    el},\pi,\alpha)$ satisfies the energy inequality  
\begin{align}
&\nonumber
\int_{\Omega}\frac\rho2|\DT{u}(\GGG T \EEE)|^2\,\d x
+\int_0^{\GGG T \EEE}\!\!\int_{\Omega}\rho \DT{u}\cdot\GGG\DDT{{u}}\EEE_{\text{\sc d}}\,\d x\,\d s
\\
&\nonumber\qquad +\int_{\Omega}\GGG \Big(\frac12\C(\alpha(\GGG T \EEE)){e}_{{\rm
    el}}(\GGG T \EEE):{e}_{{\rm el}}(\GGG T \EEE)- \phi(\alpha(\GGG T \EEE))+ \frac{\kappa}{p}|\nabla \alpha(\GGG T \EEE)|^p\Big) \EEE\,\d x\\
&\nonumber\qquad +D_{\mathcal{R}}(\alpha;\pi;0,\GGG T \EEE)+\int_0^{\GGG T \EEE}\!\!\int_{\Omega}\GGG\Big(\D(\alpha)\DT{e}_{{\rm el}}:\DT{e}_{{\rm el}}
+\eta \DT{\alpha}^2\Big)\EEE\,\d x\,\d t \\
&\nonumber \quad\le\int_{\Omega}\frac\rho2 v_0^2\,\d x  +\int_{\Omega}\GGG \Big(\frac12\C(\alpha_0)(e(u_0)-\pi_0):(e(u_0)-\pi_0)- \phi(\alpha_0)+
\frac{\kappa}{p}|\nabla \alpha_0|^p\Big)\EEE\,\d x\\
&\nonumber \qquad
+\int_{\Omega}\rho \DT{u}(\GGG T \EEE)\cdot \DT{u}_{{\text{\sc d}}}(\GGG T \EEE)\,\d x
+\int_{\Omega}\rho v_0\cdot \DT{u}_{{\text{\sc d}}}(0)\, \d x\\
&\nonumber \qquad+\int_0^{\GGG T \EEE}\!\!\int_{\Omega}\GGG\Big(\EEE \C(\alpha)e_{{\rm
    el}}:e(\DT{u}_{\text{\sc d}})  +
\D(\alpha)\DT{e}_{{\rm el}}:e(\DT{u}_{{\text{\sc d}}})
+
{f}\cdot (\DT{u}-\DT{u}_{{\text{\sc d}}})\GGG\Big)\EEE\,\d x\,\d t.\label{eq:together1fin}
\end{align}
\item[(C4)]
The quadruple $(u,e_{\rm el},\pi,\alpha)$ satisfies the damage inequality
\begin{align}\nonumber
&\nonumber\int_0^T\!\!\int_{\Omega} 
\phi'(\alpha)\varphi-\GGG \kappa \EEE |\nabla \alpha|^{p-2}\nabla\alpha\cdot\nabla\varphi
\GGG-\frac12 \EEE(\varphi{-}\DT\alpha)\C'(\alpha){e}_{{\rm el}}:{e}_{{\rm el}}-\eta \DT{\alpha}\varphi
\,\d x\,\d t\\[-.3em]
&\qquad\qquad\leq \int_{\Omega}\GGG\phi(\alpha(T))\EEE
-
\GGG\phi(\alpha_0)\EEE
-
\frac{\UUU \kappa}{p}|\nabla \alpha(T)|^p
+
\frac{\UUU \kappa}{p}|\nabla \alpha_0|^p\,\d x
-\int_0^T\!\!\int_{\Omega}
\eta\DT{\alpha}^2\,\d x\,\d t,
\end{align}
for all $\varphi\in W^{1,p}(\Omega)$ with $\varphi(x)\leq 0$ for a.e. $x\in \Omega$.
\end{itemize}
\end{definition}

The main result of the paper consists in showing existence of weak solutions to 
\eqref{system}--\eqref{initial conditions}.
Let us summarize the assumption on the data of the model:
\begin{subequations}\label{eq:hp+}\begin{align}
&\label{eq:hp-CD}
\bbC:\R\to \mathcal{L}_{\rm sym}(\mthree_{\rm sym})\ \text{ continously differentiable,
positive semidefinite, nondecreasing},
\\&\label{eq:def-D}
\bbD(\cdot)=\bbD_0+\taur\bbC(\cdot),
\D_0\ \text{ positive definite},\ \ \taur\ge0,
\\\label{eq:hp-yield}
&\phi:\R\to\R\ \text{ continuously differentiable, nondecreasing},
\\&\SYLD:\R\to\R\ \text{ continuously differentiable, positive, and nondecreasing,},
\\
&\label{eq:hp-cohesive}
\C'(0)=0,\ \ \ \phi'(0)\ge0,
\\&\label{eq:eta}\UUU
\eta \in L^\infty(\Omega), \ \ \eta \geq \eta_0 >0 \ \ \text{a.e.},
\\&\label{eq:kappa}\UUU
\kappa \in L^\infty(\Omega), \ \ \kappa \geq \kappa_0 >0 \ \
\text{a.e.},
\\&\label{eq:rho}\UUU
\rho \in L^\infty(\Omega), \ \ \rho \geq \rho_0 >0 \ \ \text{a.e.}
\end{align}\end{subequations}
where $\taur>0$ is a constant denoting \UUU a \EEE relaxation
time. \UUU The structural assumption \eqref{eq:def-D} is instrumental
in making our existence theory amenable. It arises naturally by
assuming $\bbC(\cdot)$ and $\bbD(\cdot)$ to be pure second-order
polynomials of the damage variable $\alpha$, namely
$\bbC(\alpha)=\alpha^2\bbC_2$ (recall \eqref{eq:hp-cohesive}) and
$\bbD(\alpha)=\bbD_0+\alpha^2\bbD_2$. By assuming the two tensors
$\bbC_2$ and $\bbD_2$ to be spherical, namely  $\bbC_2=c_2I_4$ and
$\bbD_2=d_2I_4$ for some $c_2,\,d_2> 0$ where $I_4$ is the identity $4$-tensor, one can define $\taur=d_2/c_2$
in order to get \eqref{eq:def-D}. \EEE
Assumption \eq{eq:hp-cohesive} ensures that $\alpha$ stays non-negative
during the evolution even if the constraint $\alpha\ge0$ is not
explicitly \UUU included \EEE  in the problem, \UUU see Remark \ref{cohe}
below. \EEE  

\begin{theorem}[Existence]
\label{thm:main}
\UUU Under assumptions \eqref{eq:hp} on initial conditions and loading
and   \eq{eq:hp+} on data there \EEE
 exists a weak solution to \eqref{system}--\eqref{initial conditions} in 
the sense of Definition \ref{def:weak}. Moreover, this solution \UUU
has the additional regularity \EEE
$(u, e_{\rm el},\pi)\in \UUU W^{1,\infty}(0,T;BD(\Omega;\R^d))\EEE$
$\times W^{1,\infty}(0,T;L^2(\Omega;\mthree_{\rm sym}))\times W^{1,\infty}(0,T;\mb)$.
\end{theorem}

The proof of Theorem \ref{thm:main} is postponed to Section
\ref{sec:proof-main},  \UUU where we present a \EEE conceptually implementable, numerically stable, and convergent
numerical algorithm.
Instead, we conclude this section with some final remarks.

\begin{remark}[Body and surface loads]
As pointed out in \cite[Introduction]{babadjian.mora}, for quasistatic 
evolution in perfect plasticity one has to impose a compatibility condition 
between body and surface loads, namely a \emph{safe load} to ensure that the 
body is not in a free flow. In the dynamic case, under the assumption of null 
surface loads, this condition can be weaken for what concerns body loads; 
see,{ \it{e.g.}}, \cite{kisiel}.
\end{remark}

\begin{remark}[Cohesive damage assumption]\label{cohe}
We will not include in the model reaction forces to the constraint $0\le\alpha\le1$. 
This would be encoded by rewriting \eqref{system-alpha} as
$$
\partial
\zeta(\DT\alpha)+\frac12 \C'(\alpha)e_{\rm el}:e_{\rm el}
+p_\text{\sc r}\ni \phi'(\alpha)+\UUU {\rm div}\, (\kappa |\nabla
\alpha|^{p-2} \nabla \alpha) \EEE
\quad\text{ where }\ \ \ p_\text{\sc r}\in N_{[0,1]}(\alpha)\,;
$$
here $N_{[0,1]}(\cdot)$ denote the normal cone and $p_\text{\sc r}$ is a 
``reaction pressure'' to the constraints $0\le\alpha\le1$. We point out that 
the presence of this additional term in the damage flow rule would cause a 
loss of regularity for the damage variable. In order to avoid such problem 
we will rather enforce the constraint $0\leq \alpha\leq 1$ by exploiting the 
irreversibility of damage, and by restricting our analysis to the situation 
in which the damage is cohesive.
\end{remark}

\begin{remark}[Regularity of $\Gamma$]
We remark that the $C^2$-regularity of $\Gamma$ is needed in order to 
apply \cite[Proposition 2.5]{kohn.temam}, and define a duality between stresses 
and plastic strains. For $d=2$, owing to the results in \cite{francfort.giacomini}, 
it is also possible to analyze the setting in which $\Gamma$ is Lipschitz. The same 
strategy can not be applied for $d=3$, for it would require ${\rm div}\,\sigma\in L^3(\Omega)$, 
whereas here we can only achieve ${\rm div}\,\sigma\in L^2(\Omega)$.
\end{remark}

\begin{remark}[The role of the term $\eta\DT{\alpha}$]
The term $\eta\DT{\alpha}$ in \eqref{system-alpha} guarantees strong convergence of the 
damage-interpolants in the time-discretization scheme to the limit damage variable. 
This, in turn, is a key point to ensure strong convergence of the elastic stresses, 
which is fundamental for the proof of the damage inequality in condition (C4). From a 
modeling point of view, this might be interpreted as some additional dissipation 
related with the speed of \UUU the \EEE damaging process contributing
to the heat production, \UUU possibly leading to an increase of
temperature. \EEE  The \GGG microscopical \EEE
idea behind it is that faster mechanical processes cause \UUU higher
heat production and therefore higher dissipation.  \EEE
\end{remark}

\begin{remark}[Phase-field fracture]
Our cohesive damage with $\bbC'(0)=0$ has the drawback that,
while $\alpha$ approaches zero, the driving force needed for its evolution
rises to infinity. \UUU This model is anyhow \EEE used in the phase-field 
approximation of fracture. 
\begin{align}
&\!\!\!\!\!\!
\mathscr{E}(u,\alpha):=\int_\Omega
\frac{
(\varepsilon/\varepsilon_0)^2{+}\alpha^2}2\bbC_1 e(u){:}e(u)+
\!\!\lineunder{
G_{\rm c}\Big(\frac{1}{2\varepsilon}(1{-}\alpha)^2\!
+\frac\varepsilon2|\nabla\alpha|^2\Big)}{crack surface density}\!\,\d x
\label{eq6:AM-engr+}
\end{align}
with with $G_{\rm c}$ denoting the energy of fracture
and with $\varepsilon$ controlling a ``characteristic'' width of the 
phase-field fracture zone(s). The physical dimension of $\varepsilon_0$ 
as well as of $\varepsilon$ is m (meters) while the physical dimension 
of $G_{\rm c}$ is J/m$^2$. In the model \eq{def-of-E}, 
it means $\bbC(\alpha)=(\varepsilon^2/\varepsilon_0^2{+}\alpha^2)\bbC_1$
and $\phi(\alpha)=-G_{\rm c}(1{-}\alpha)^2/(2\varepsilon)$ while $\kappa=\varepsilon$
and $p=2$. This is known as the so-called \emph{Ambrosio-Tortorelli 
functional}\index{Ambrosio-Tortorelli functional}.
Its motivation came from the static case, where 
this approximation was proposed by Ambrosio and Tortorelli 
\cite{AmbTor90AFDJ,AmbTor90AFDP} and the asymptotic 
analysis for $\varepsilon\to0$ was rigorously \UUU proved \EEE
first for the scalar-valued case.
The generalization for the vectorial case is due to Focardi \cite{Foca01VAFD}.
Later, it was extended \UUU to the \EEE evolution situation, namely for 
a rate-independent cohesive damage, in \cite{Giac05ATAQ}, see 
also \cite{BoFrMa08VAF,BoLaRi11TDMD,caponi2018,LaOrSu10ESRM,MieRou15RIST}
where \GGG  inertial forces are incorporated in the description. \EEE \UUU Note however
that \EEE plasticity was not involved in all these references.
Some \GGG modifications have \EEE been \UUU addressed \EEE in \cite{BMMS14MPCC}, see
also \cite{Roub??MDDP} for various other \GGG models, and \cite{conti.focardi.iurlano2016,focardi.iurlano2014,freddi.iurlano2017,iurlano2014} for the linearized and cohesive-fracture settings.\EEE
\end{remark}

\begin{remark}[Ductile damage/fracture]
A combination of damage/fracture with plasticity is sometimes 
denoted by the adjective ``ductile'', in contrast to ``brittle'',
if plasticity is not considered. There are various scenarios 
of combination of plastification processes with damage, that can
model various phenomena in fracture mechanics. \UUU Here, we address
the case of damage-dependent elastic response and the yield stress. \EEE
\end{remark}

\begin{remark}[Influence of damage on the energy equality]\label{rk:no-en-cons}
We point out that, in the absence of damage, energy conservation \UUU
could be recovered. Indeed, it would be possible to prove the energy
\EEE  equality, \UUU which \EEE would \UUU then ensure the validity of \EEE
the flow rule \eqref{system-p} \UUU as well. \EEE A detailed analysis
of \UUU an analogous albeit quasistationary case \EEE has been performed 
in \cite[Section 6]{dalmaso.desimone.mora06} in the quasistatic framework. An adaptation 
of the argument yields the analogous statements in the dynamic setting.
\end{remark}

\section{Some formal calculus first}\label{sec:formal}
We first highlight a formal strategy 
that will lead to the proof of Theorem \ref{thm:main}, avoiding (later necessary) 
technicalities. In particular, we first derive the energetics of the model 
by performing some standard tests of \eqref{system}
against the time derivatives $(\DT u,\DT \pi,\DT \alpha)$.
Further a-priori estimates will be obtained by performing a test of the same equations 
against higher-order time-derivatives of the maps. \UUU Eventually, a
direct strong-convergence argument will be presented. \EEE  

All the arguments will be eventually made rigorous in 
Sections~\ref{sec-apriori}--\ref{sec:proof-main}
by means of a time-discretization procedure, combined with a passage to the limit as 
the time-step vanishes. The estimates described in Subsections~\ref{subs:higher-order}--\ref{subs:conv} will 
be essential to pass to the limit in the time-discrete damage equation.

\subsection{Energetics of the model and first estimates}
\label{subs:energetics}
A formal test of the equations/inclusion \eqref{system} successively 
against $\DT{u}$, $\DT{\pi}$, and $\DT{\alpha}$
yields
\begin{subequations}\begin{align}
&\label{eq:1-formal} \int_{\Omega} \GGG\Big(\EEE \rho\DDT{u}(t)\cdot \DT{u}(t)
+
\sigma(t): e(\DT{u}(t))\GGG\Big)\EEE\,\d x=\int_{\Omega} f(t)\cdot \DT{u}(t)\,\d x
+\int_{\Gamma}\sigma(t) \nu_\Gamma\cdot\GGG\DT{u}_{\text{\sc d}}(t)\EEE\,\d\HH^{d-1},
\\&\label{eq:2-formal}
\int_{\Omega}{\rm dev}\,\sigma(t):\DT{\pi}(t)\,\d x=\int_{\Omega}
\SYLD(\alpha(t))|\DT\pi(t)|\,\d x,
\\\nonumber
&\int_{\Omega} \eta \DT{\alpha}(t)^2\,\d x=\int_{\Omega} \GGG\Big(\EEE\phi'(\alpha(t))\DT{\alpha}(t)
\\[-.5em]
\label{eq:3-formal}
&\qquad\qquad\qquad\qquad
-\frac12 
\C'(\alpha(t))\DT{\alpha}(t) e_{\rm el}(t):e_{\rm el}(t)
-
\UUU \kappa \EEE |\nabla \alpha(t)|^{p-2}\nabla \alpha(t)\cdot \nabla \DT{\alpha}(t)\GGG\Big)\EEE\,\d x.
\end{align}\end{subequations}
Integrating \eqref{eq:1-formal} in time, by \eqref{initial conditions}, \eqref{eq:2-formal}, 
and by the definition of $e_{\rm el}$, we obtain
\begin{align}\nonumber 
&\int_{\Omega}\GGG\Big(\EEE \frac\rho2|\DT{u}(t)|^2
+
\frac{1}{2}\C(\alpha(t)) e_{\rm el}(t): e_{\rm el}(t)\GGG\Big)\EEE\,\d x
-\int_0^t\!\!\int_{\Omega}\frac12\C'(\alpha) \DT{\alpha} e_{\rm el}:e_{\rm el}\,\dx\,\ds\\
&\nonumber \qquad 
+\int_0^t\!\!\int_{\Omega}\D(\alpha)\DT{e}_{\rm el}:\DT{e}_{\rm el}\,\d x\,\d s
+\int_0^t\!\!\int_{\Omega}
\SYLD(\alpha)|\DT{\pi}|\,\d x\,\d s\\
&\quad= \int_{\Omega}\GGG\Big(\EEE \frac{\rho}{2}|v_0|^2
+\frac12
\C(\alpha_0)e_{\rm el}(0):e_{\rm el}(0)\GGG\Big)\EEE\,\dx
+\int_0^t\!\!\int_{\Omega} f\cdot \DT{u}\,\d x\,\d s
+\int_0^{t}\!\!\int_{\Gamma_\text{\sc D}}\!\!\sigma\nu_\Gamma^{}\cdot \GGG\DT{u}_{\text{\sc d}}\EEE\,\d\HH^{d-1}\,\d s.
\label{eq:1-formal-int}
\end{align}
In view of \eqref{bc} and \eqref{initial conditions}, an integration in time 
of \eqref{eq:3-formal} yields
\begin{align}
\int_0^{t}\!\int_{\Omega}\GGG\Big(\EEE\!\eta \DT{\alpha}^2
+\frac12 
\DT{\alpha}\C'(\alpha) e_{\rm el}:e_{\rm el}\GGG\Big)\EEE\,\d x\,\d s
+\int_{\Omega}\GGG\Big(\EEE \frac{\UUU \kappa}{p}|\nabla\alpha(t)|^p
-
\phi(\alpha(t))\GGG\Big)\EEE\,\d x
=\!\int_{\Omega}\GGG\Big( \frac{\kappa}{p}|\nabla\alpha_0|^p\!\EEE-\phi(\alpha_0)\GGG\Big)\EEE\,\d x.
\label{eq:3-formal-int}\end{align}
Thus, summing \eqref{eq:1-formal-int} and \eqref{eq:3-formal-int}, by \eqref{bc} we deduce the energy equality \eqref{eq:formal-en-equality}.

To see the energy-based estimates from \eqref{eq:formal-en-equality},
here we should use the Gronwall inequality for the term $f\cdot\DT u$ benefitting from
having the kinetic energy on the left-hand side, and the by-part integration
of the Dirichlet loading term. \GGG We stress that the last term in \eqref{eq:1-formal-int} can be rigorously defined as in \eqref{eq:def-bd-term}. \EEE
This way, we can see the estimates
\begin{subequations}\begin{align}\label{basic-est-u}
&u\in L^{\infty}(0,T;BD(\Omega;\R^d))\cap W^{1,\infty}(0,T;L^2(\Omega;\R^d)),&\\
&e_{\rm el}\in H^1(0,T;L^2(\Omega;\mthree_{\rm sym})),&\\
&\pi\in BV(0,T;\mb),
\\&\alpha\in L^\infty(0,T;W^{1,p}(\Omega))\cap H^1(0,T;L^2(\Omega)).
\end{align}\end{subequations}
Unfortunately, these estimates do not suffice for the convergence
analysis \GGG as the time step \UUU goes to $0$. \EEE 
In particular, \UUU in relation \eq{strong-conv-e-perfect}  later on
one needs to handle the term  $\rho\DDT u_k\cdot\DT u$, which is still
not integrable under \eqref{basic-est-u}. \EEE

\subsection{Higher-order tests}\label{subs:higher-order}
In this subsection we perform an extension of the regularity estimate in Subsection \ref{subs:energetics}, relying on the unidirectionality of the damage evolution, on the fact that $\SYLD(\cdot)$ is 
nondecreasing, and on the monotonicity of $\C(\cdot)$ with respect to the L\"owner ordering.
We introduce the abbreviation
\begin{equation}
\label{eq:def-w-test}
w:=u+\taur \DT{u},\quad \varepsilon_{{\rm el}}:=e_{{\rm el}}+\taur\DT{e}_{{\rm el}},
\quad \text{and}\quad\varpi=\pi+\taur\DT\pi,
\end{equation}
and observe that, $\DDT u=(\DT w{-}\DT u)/{\taur}.$
Hence, the equilibrium equation rewrites as
\begin{equation}
\label{eq:new-eq}
\rho \frac{\DT w}{\taur}-{\rm div }\,\sigma=f+\rho\frac{\DT u}{\taur}.
\ee

We first argue by testing the plastic flow rule \eqref{system-p} against 
$\DT\varpi$. We use the (here formal) calculus 
\begin{align}\label{formal-calculus}
\SYLD(\alpha){\rm Dir}(\DT \pi)\Colon\DDT \pi
=\frac{\pl}{\pl t}\Big(\SYLD(\alpha)|\DT \pi|\Big)-\DT\alpha\SYLD'(\alpha)
|\DT \pi|\ge\frac{\pl}{\pl t}\Big(\SYLD(\alpha)|\DT\pi|\Big)
\end{align}
because $\DT\alpha\SYLD'(\alpha)|\DT\pi|\le0$ when assuming 
$\SYLD(\cdot)$ nondecreasing and using $\DT\alpha\le0$, cf.\ \eq{eq:monoton2}.
This formally yields
\begin{align}\nonumber
&\int_0^T\!\!\int_{\Omega}
\SYLD(\alpha)|\DT{\pi}|\,\dx\,\d t
+\taur\int_{\Omega} \GGG\Big(\EEE
\SYLD(\alpha(t))|\DT\pi(t)|
-
\taur
\SYLD(\alpha_0)|\DT{\pi}(0)|\GGG\Big)\EEE\,\d x
\\& \nonumber\qquad
=\int_0^T\!\!\int_{\Omega}\SYLD(\alpha)|\DT \pi|\,\dx\,\d t
+\taur\int_{\Omega} \SYLD(\alpha(t))|\DT \pi(t)|\,\d x-\taur \int_{\Omega}\SYLD(\alpha(0))|\DT \pi(0)|\,\d x
\\& \qquad\qquad\qquad\qquad\qquad\qquad\qquad\qquad\qquad\qquad\qquad\qquad\qquad
\leq \int_0^T\!\!\int_{\Omega}\sigma\Colon
\DT\varpi\,\dx\,\d t.
\label{estimate-p}
\end{align}

\GGG Analogously, \EEE testing \eqref{eq:new-eq} against $\DT{w}$ and integrating in time, 
by \eqref{bc} we deduce
\begin{align}
&
\int_0^T\!\!\int_{\Omega}\GGG\Big(\EEE \frac{\rho}{\taur}|\DT{w}|^2
+
\sigma:e(\DT{w})\GGG\Big)\EEE\,\d x\d t=\int_0^T\!\!\int_{\Omega}\GGG\Big(\EEE f\cdot\DT{w}
+
\frac{\rho}{\taur}\DT{u}\cdot\DT{w}\GGG\Big)\EEE\,\d x\d t
+\int_0^T\!\!\int_{\Gamma_{\text{\sc d}}}\!\!\sigma\nu_\Gamma^{}\cdot \GGG(\DT{u}_{\text{\sc d}}
+\chi \DDT{u}_{\text{\sc d}})\EEE\,\d \HH^{d-1}\GGG \d t\EEE.\label{eq:ho1}
\end{align}
By the definition of the tensor $\D$ (see Subsection \ref{subs:el-visc}), and 
by \eqref{eq:monoton1}, we infer that
\begin{align}\nonumber
&\int_0^T\!\!\int_{\Omega}\sigma:e(\DT{w})\,\d x\,\d t
=\int_0^T\!\!\int_{\Omega}\GGG\Big(\EEE\C(\alpha) \varepsilon_{{\rm el}}:\DT\varepsilon_{{\rm el}}
+
\D_0\DT{e}_{{\rm el}}: \DT\varepsilon_{{\rm el}}
+
\sigma:\DT{\varpi}\GGG\Big)\EEE\,\d x\,\d t\\
&\nonumber\quad\geq  \int_{\Omega}\frac12\C(\alpha(t))\varepsilon_{{\rm el}}(t):\varepsilon_{{\rm el}}(t)\,\dx+\int_0^T\int_{\Omega}\D_0\DT{e}_{{\rm el}}: \DT{e}_{{\rm el}}\,\dx \,\d t
-\int_{\Omega}\frac12\C(\alpha_0)\varepsilon_{{\rm el}}(0):\varepsilon_{{\rm el}}(0)\,\dx
\\&\qquad
+\frac{\taur}{2} \int_{\Omega}\D_0\DT{e}_{{\rm el}}(t): \DT{e}_{{\rm el}}(t)\,\dx
-\frac{\taur}{2} \int_{\Omega}\D_0\DT{e}_{{\rm el}}(0): \DT{e}_{{\rm el}}(0)\,\dx+\int_0^T \int_{\Omega}\sigma:\DT{\varpi}\,\d x\,\d t.
\label{eq:ho2}\end{align}
Thus, by combining \eqref{estimate-p}, with \eqref{eq:ho1} and \eqref{eq:ho2}, we obtain 
the inequality
\begin{align*}
&\frac{1}{\taur}\int_0^T\!\!\int_{\Omega}\rho|\DT{w}|^2\,\d x\,\d t+\frac12 \int_{\Omega}\C(\alpha(t))\varepsilon_{{\rm el}}(t):\varepsilon_{{\rm el}}(t)\,\dx\\
&\qquad+\int_0^T \int_{\Omega}\D_0\DT{e}_{{\rm el}}: \DT{e}_{{\rm el}}\,\dx \,\d t
+\frac{\taur}{2} \int_{\Omega}\D_0\DT{e}_{{\rm el}}(t): \DT{e}_{{\rm el}}(t)\,\dx+\int_0^T\!\!\int_{\Omega}
\SYLD(\alpha)|\DT\pi|
\,\dx\,\d t\\
&\qquad +\taur\int_{\Omega} 
\SYLD(\alpha(t))|\DT\pi(t)|
\,\d x\leq \frac12 \int_{\Omega}\C(\alpha_0)\varepsilon_{{\rm el}}(0):\varepsilon_{{\rm el}}(0)\,\dx\\
&\qquad+\frac{\taur}{2} \int_{\Omega}\D_0\DT{e}_{{\rm el}}(0): \DT{e}_{{\rm el}}(0)\,\dx+\taur\int_{\Omega} 
\SYLD(\alpha_0)|\DT\pi(0)|\,\d x\\
&\qquad +\int_0^T\!\!\int_{\Omega}f\cdot\DT{w}\,\dx\d t+\int_0^T\!\!\int_{\Gamma_{\text{\sc d}}}\!\!\sigma\nu_\Gamma^{}\cdot \GGG(\DT{u}_{\text{\sc d}}
+\chi \DDT{u}_{\text{\sc d}})\EEE\,\d \HH^{d-1}\GGG \d t\EEE+\frac{1}{\taur}\int_0^T\!\!\int_{\Omega}\rho\DT{u}\cdot\DT{w}\,\dx\d t.
\end{align*}
 Let us note that \UUU we can use \eq{basic-est-u} in order to control
 $\DT u$ in the last term above. As for initial data, we need here
 that \EEE $\DT{e}_{{\rm el}}(0)\in L^2(\Omega;\mthree_{\rm sym})$ and 
$\DT{\pi}(0)\in L^2(\Omega;\mthree_{\rm dev})$, \UUU which follows
under the provisions of \eq{eq:hp}. \EEE
Eventually, by \eqref{eq:formal-en-equality}, and \eqref{eq:def-w-test} this yields the following 
additional regularity for the displacement, and for the elastic and plastic strains
\begin{subequations}\begin{align}
&u\in W^{1,\infty}(0,T;BD(\Omega;\R^d))\cap H^2(0,T;L^2(\Omega;\R^d)),&\\
&e_{\rm el}\in W^{1,\infty}(0,T;L^2(\Omega;\mathbb{M}^{d\times d})),&\\
&\pi\in W^{1,\infty}(0,T;\mb).
\end{align}\end{subequations}

\subsection{One more estimate for \UUU the \EEE strong convergence of $e_{\rm el}$'s}\label{subs:conv}
\UUU The \EEE 
strong convergence of the elastic strains $e_{\rm el}$ \UUU is \EEE needed 
for the limit passage in the damage flow rule. \UUU The failure of
\EEE energy conservation \GGG(see Remark \ref{rk:no-en-cons}) \EEE prevents the usual ``limsup-strategy'', but one can 
estimate directly the difference
between the (presently \UUU still \EEE unspecified) approximate solution $(u_k,\pi_k)$ 
and its limit $(u,\pi)$ \GGG punctually \EEE as:
\begin{align}\nonumber
&\int_Q\bbD(\alpha_k)(\DT e_{{\rm el},k}{-}\DT e_{\rm el})\Colon(\DT e_{{\rm el},k}{-}\DT e_{\rm el})\,\d x\d t
\\[-.5em]&\nonumber\qquad\qquad\qquad
+\int_\Omega\frac12\bbC(\alpha_k(T))(e_{{\rm el},k}(T){-}e_{\rm el}(T))
\Colon(e_{{\rm el},k}(T){-}e_{\rm el}(T))\,\d x
\\&\nonumber
\le\int_Q\int_\Omega\big(\bbD(\alpha_k)(\DT e_{{\rm el},k}{-}\DT e_{\rm el})
+\bbC(\alpha_k)(e_{{\rm el},k}{-}e_{\rm el})\big)\Colon(\DT e_{{\rm el},k}{-}\DT e_{\rm el})\,\d x\d t
\\&\nonumber\le\int_Q\bigg((f{-}\rho\DDT u_k)\Cdot(\DT u_k{-}\DT u)
-\big(\bbD(\alpha_k)\DT e_{\rm el}+\bbC(\alpha_k)e_{\rm el}\big)\Colon(\DT e_{{\rm el},k}{-}\DT e_{\rm el})
\\[-.5em]&\qquad\qquad\qquad
+\SYLD(\alpha_k)(|\DT\pi|-|\DT\pi_k|)\bigg)\,\d x\d t.
\label{strong-conv-e-perfect}
\end{align}
The first inequality in 
\eq{strong-conv-e-perfect} is due to \UUU the \EEE monotonicity of $\bbC(\cdot)$ with 
respect to the L\"owner ordering so that, due to  \eq{eq:monoton1}, it holds
\begin{align}\nonumber
&\frac{\pl}{\pl t}\left(
\frac12\bbC(\alpha_k)(e_{{\rm el},k}{-}e_{\rm el}){:}(e_{{\rm el},k}{-}e_{\rm el})
\right)\nonumber\\
& \UUU = \frac12\DT\alpha_k\bbC'(\alpha_k)(e_{{\rm el},k}{-}e_{\rm
  el}){:}(e_{{\rm el},k}{-}e_{\rm el}) +
 \bbC(\alpha_k) \frac{\pl}{\pl t}\GGG\left( \frac12(e_{{\rm
      el},k}{-}e_{\rm el}){:}(e_{{\rm
      el},k}{-}e_{\rm el})\right)  \EEE \nonumber\\
&\UUU \leq  \EEE \bbC(\alpha_k)(e_{{\rm el},k}{-}e_{\rm el}){:}(\DT e_{{\rm el},k}{-}\DT e_{\rm el})\,.
\end{align}
while the second \GGG step \EEE in \eq{strong-conv-e-perfect} is due to 
the inequality
${\rm dev}\,\sigma\Colon(\DT\pi_k-\DT\pi)\ge\SYLD(\alpha_k)(|\DT\pi_k|
-|\DT\pi|)$, \UUU following \EEE from the plastic flow rule 
$\SYLD(\alpha_k){\rm Dir}(\DT\pi_k)\ni{\rm dev}\,\sigma$, with $\sigma$ from 
\eq{system-u-new}. 

\UUU By using weak* upper semicontinuity and the uniform convergence  $\alpha_k\to\alpha$ one checks  that the limit
superior of the right-hand side in \eq{strong-conv-e-perfect} can be \EEE estimated from above by zero 
(so that, in fact, the limit does exist and equals to zero). \GGG We refer to Proposition \ref{prop:strong-convergence} for the rigorous implementation of the above argument. \EEE

\section{Staggered two-step time-discretization scheme}\label{sec:time-discrete}
This section is devoted to the \UUU solution of \EEE a discrete counterpart 
of the system of equations \eqref{system}--\eqref{initial conditions}, and to the 
proof of a-priori estimates for the associated piecewise constant, piecewise affine, 
and piecewise quadratic in-time interpolants.

Fix $n\in\N$, set $\tau:=T/n$, and consider the equidistant time partition of the interval 
$[0,T]$ with step $\tau$. We define the discrete body-forces by setting 
$f_{\tau}^k:=\int_{(k-1)\tau}^{k\tau}f(t)\,\d t$ for all
$k\in\{1,\dots,{T}/{\tau}\}$. \UUU We consider \EEE the following 
time-discretization scheme: 
\begin{subequations}\label{eq:form-disc}
\begin{align}
&\label{eq:form-disc-equilibrium}\rho \delta^2 u_{\tau}^k-{\rm div }\big(\C(\alpha_{\tau}^{k-1})e_{{\rm el},\tau}^k+\D(\alpha_{\tau}^{k-1})\delta e_{{\rm el},\tau}^k\big)=f_{\tau}^k,\\
&\label{eq:form-disc-flow-rule}
\SYLD(\alpha_{\tau}^{k-1}){\rm Dir}(\delta \pi_{\tau}^k)
\ni{\rm dev}\big(\C(\alpha_{\tau}^{k-1})e_{{\rm el},\tau}^k+\D(\alpha_{\tau}^{k-1})\delta e_{{\rm el},\tau}^k\big),\\
&\label{eq:form-disc-damage}\partial 
\zeta(\delta\alpha_{\tau}^k)+\frac12 \C^\circ(\alpha_{\tau}^k,\alpha_{\tau}^{k-1})e_{{\rm el},\tau}^{k}:e_{{\rm el},\tau}^{k}\ni 
\phi^\circ(\alpha_{\tau}^k,\alpha_{\tau}^{k-1})
+\UUU {\rm div} \, (\kappa|\nabla \alpha_{\tau}^k |^{p-2}\nabla \alpha_{\tau}^k ),
\end{align}\end{subequations}
\UUU to be complemented with the boundary conditions 
\begin{subequations}\label{eq:bc-disc}
  \begin{align}
   &\left( \C(\alpha_\tau^{k-1})  e_{{\rm el},\tau}^k + \D(\alpha_\tau^{k-1}) \delta e_{{\rm el},\tau}^k\right) \nu_\Gamma = 0 \quad \text{on} \ \
   \Gamma_{{\text{\sc n}}}  \label{eq:bc-disc1}\\
   & \kappa|\nabla \alpha_{\tau}^k |^{p-2}\nabla \alpha_{\tau}^k \cdot
   \nu_\Gamma =0 \quad \text{on} \ \
   \Gamma .\label{eq:bc-disc2}
  \end{align}
\end{subequations}
\EEE
\UUU Here, \EEE $\delta$ and $\delta^2$ denote the first and second order 
finite-difference operator, that is 
$$
\delta u_{\tau}^k:=\frac{u_{\tau}^k-u_{\tau}^{k-1}}{\tau}
\quad\text{and}\quad
\delta^2 u_{\tau}^k:=\delta\big[\delta u_{\tau}^k\big]=
\frac{u_{\tau}^k-2u_{\tau}^{k-1}+u_{\tau}^{k-2}}{\tau^2},
$$
and where the tensor $\C^\circ
(\alpha,\tilde\alpha)$ \UUU and the scalar $\phi^\circ(\alpha ,\tilde
\alpha)$ are defined for all $\alpha, \, \tilde \alpha \in \R$ \EEE  as
\begin{align*}
&\C^\circ
(\alpha,\tilde\alpha)
:=\begin{cases}
\displaystyle{\frac{\C(\alpha){-}\C(\tilde\alpha)}{\alpha-\tilde\alpha}}
&\UUU \text{if }\ \alpha\not = \tilde \alpha 
\\
\C'(\alpha)=\C'(\tilde\alpha)
&\UUU \text{if }\ \alpha=\tilde\alpha,
\end{cases}\\
& 
\phi^\circ
(\alpha,\tilde\alpha)
:=\begin{cases}
\displaystyle{\frac{\phi(\alpha){-}\phi(\tilde\alpha)}{\alpha-\tilde\alpha}}
&\text{if }\ 
\alpha\ne\tilde\alpha\,,\\
\phi'(\alpha)=\phi'(\tilde\alpha)
&\text{if }\ \alpha=\tilde\alpha\,.
\end{cases}
\end{align*}
Let us note that, if  $\phi(\cdot)$ is affine, then simply 
 $\phi^\circ(\alpha_{\tau}^k,\alpha_{\tau}^{k-1})=\phi'$.
Similarly, if $\C(\cdot)$ were affine, then $\C^\circ(\alpha_{\tau}^k,\alpha_{\tau}^{k-1})=\C'$. \GGG We point out that \EEE this case would be in conflict with \eq{eq:hp-cohesive} unless $\bbC$ would
be independent of damage.



\subsection{Weak solutions to the time-discretization scheme}
\label{subs:weak}
In order to define a notion of weak solutions to \eqref{eq:form-disc-flow-rule},
we need to preliminary introduce a duality between stresses and plastic 
strains. We work along the footsteps of \cite{kohn.temam} and 
\cite[Subsection 2.3]{dalmaso.desimone.mora06}. We first define the set
\be{eq:def-sigma-omega}
\Sigma(\Omega):=\big\{\sigma\in L^2(\Omega;\mthree_{\rm sym}):\ 
{\rm dev}\,\sigma\in L^{\infty}(\Omega;\mthree_{\rm dev})\ \text{ and }\ 
{\rm div}\,\sigma\in L^2(\Omega;\R^d)\big\}.
\ee
By \cite[Proposition 2.5 and Corollary 2.6]{kohn.temam}, for every 
$\sigma\in\Sigma(\Omega)$ there holds
$$
\sigma\in L^{\infty}(\Omega;\mthree_{\rm sym}).
$$
In addition, we can introduce the trace 
$[\sigma\nu_\Gamma^{}]\in H^{-1/2}(\Gamma;\R^d)$ (see e.g.\ 
\cite[Theorem 1.2, Chapter I]{temam}) as
\begin{equation}\label{eq:def-bd-term}
\scal{[\sigma\nu_\Gamma^{}]}{\psi}_{\Gamma}:=\int_{\Omega}{\rm div}\,\sigma\cdot\psi\,\dx+\int_{\Omega}\sigma:e(\psi)\,\dx
\end{equation}
for every $\psi\in H^1(\Omega;\R^d)$. Defining the normal and the tangential part of $[\sigma \nu_\Gamma^{}]$ as
$$
[\sigma \nu_\Gamma^{}]_{\nu}:=([\sigma\nu_\Gamma^{}]\cdot\nu_\Gamma^{})\nu_\Gamma^{}\quad\text{and} \quad [\sigma \nu_\Gamma^{}]_{\nu}^{\perp}:=[\sigma\nu_\Gamma^{}]-([\sigma\nu_\Gamma^{}]\cdot\nu_\Gamma^{})\nu_\Gamma^{},
$$
by \cite[Lemma 2.4]{kohn.temam} we have that $[\sigma \nu_\Gamma^{}]_{\nu}^{\perp}\in L^{\infty}(\Gamma;\R^d)$, and
$$
\|[\sigma \nu_\Gamma^{}]_{\nu}^{\perp}\|_{L^{\infty}(\Gamma;\R^d)}\leq \frac{1}{\sqrt{2}}\|{\rm dev}\,\sigma\|_{L^{\infty}(\Omega;\mthree_{\rm dev})}.
$$
Let $\sigma\in \Sigma(\Omega)$ and let $u\in BD(\Omega;\R^d)\cap L^2(\Omega;\R^d)$, with ${\rm div }\,u\in L^{2}(\Omega)$. 
We define the distribution $[{\rm dev}\,\sigma:{\rm dev}\,e(u)]$ on $\Omega$ as
\be{eq:def-meas-ed}\scal{[{\rm dev}\,\sigma:{\rm dev}\,e(u)]}{\varphi}:=-\int_{\Omega}\varphi\,{\rm
  div}\,\sigma \cdot  u\,\dx-\frac1d\int_{\Omega}\varphi\, {\rm  tr
}\,\sigma \cdot{\rm div}\,u\,\dx-\int_{\Omega}\sigma: (u\odot\nabla\varphi)\,\dx\ee
for every $\varphi\in C^{\infty}_c(\Omega)$. By \cite[Theorem 3.2]{kohn.temam} it follows that $[{\rm dev}\,\sigma:{\rm dev}\,e(u)]$ is a bounded Radon measure on $\Omega$, whose variation satisfies
$$
|[{\rm dev}\,\sigma:{\rm dev}\,e(u)]|\leq \|{\rm dev}\,\sigma\|_{L^{\infty}(\Omega;\mthree_{\rm dev})}|{\rm dev}\,e(u)|\quad\text{in }\Omega.
$$

Let $\Pi_{\Gamma_{\text{\sc d}}}(\Omega)$ be the set of admissible plastic strains, namely the set of maps $\pi\in\mb$ such that there exist $u\in BD(\Omega;\R^d)\cap L^2(\Omega;\R^d)$, $e\in L^2(\Omega;\mthree_{\rm sym})$, and $w\in W^{1,2}(\Omega;\R^d)$ with $(u,e,\pi)\in\pdeor(w)$. Note that the additive decomposition $e(u)=e+\pi$ implies that ${\rm div }\,u\in L^2(\Omega)$. 

It is possible to define a duality between elements of $\Sigma
(\Omega)$ and $\Pi_{\Gamma_{\text{\sc d}}}(\Omega)$. To be precise,
given $\pi\in \Pi_{\Gamma_{\text{\sc d}}}(\Omega)$, and $\sigma\in
\Sigma(\Omega)$, we fix $(u,e,w)$ such that $(u,e,\pi)\in\pdeor(w)$,
with $u\in L^2(\Omega;\R^d)$, and we define the measure $[{\rm
  dev}\,\sigma:\UUU \pi \EEE]\in\mb$ as
$$[{\rm dev}\,\sigma:\pi]:=\begin{cases}
[{\rm dev}\,\sigma:{\rm dev}\,e(u)]-{\rm dev}\,\sigma:\,{\rm dev}\,e&\text{in }\Omega\\
[\sigma\nu_\Gamma^{}]_{\nu}^{\perp}\cdot(w-u)\HH^{d-1}&\text{on }\Gamma_{\text{\sc d}},
\end{cases}
$$
so that
$$
\int_{\Omega\cup\Gamma_{\text{\sc d}}}\varphi\,\d[{\rm dev}\,\sigma:\pi]=\int_{\Omega}\varphi\,\d[{\rm dev}\,\sigma:{\rm dev}\,e(u)]
-\int_{\Omega}\varphi\,{\rm dev}\,\sigma:\,{\rm dev}\,e\,\dx+\int_{\Gamma_{\text{\sc d}}}\varphi[\sigma\nu]_{\nu}^{\perp}\cdot(w-u)\,\d\HH^{d-1}
$$
for every $\varphi\in C(\bar{\Omega})$. Arguing as in \cite[Section 2]{dalmaso.desimone.mora06}, one can prove that 
the definition of $[{\rm dev}\,\sigma:\pi]$ is independent of the choice of $(u,e,w)$, and that if 
${\rm dev}\,\sigma\in C(\bar{\Omega};\mthree_{\rm dev})$ and $\varphi\in C(\bar{\Omega})$, then
$$
\int_{\Omega\cup\Gamma_{\text{\sc d}}}\varphi\,\d[{\rm dev}\,\sigma:\pi]=\int_{\Omega\cup\Gamma_{\text{\sc d}}}\varphi\,{\rm dev}\,\sigma:\,\d \pi.
$$

We are now in a position to state the definition of weak solutions to the time-discretization scheme.

\begin{definition}[Weak discrete solutions]
For every $k\in\{1,\dots,{T}/{\tau}\}$, a quadruple $(u^k_{\tau},e^k_{{\rm el},{\tau}},\pi^k_{\tau},\alpha_{\tau}^k)$ is a weak solution to 
\eqref{eq:form-disc}
 if $(u_{\tau}^k,e_{\tau}^k,\pi_{\tau}^k)\in \mathscr{A}(u_{\text{\sc
     d},{\tau}}^k)$, $\alpha_{\tau}^k\in W^{1,p}(\Omega)\cap
 L^{\infty}(\Omega)$ satisfies $0\leq \alpha_{\tau}^k\leq 1$, the
 quadruple solves \eqref{eq:form-disc-damage} \GGG and \eqref{eq:bc-disc}, \GGG property \EEE
 \eqref{eq:form-disc-equilibrium} \GGG holds \UUU almost everywhere, \EEE and the following discrete flow-rule is fulfilled
\begin{equation}
\label{eq:discrete-flow-rule}
[{\rm dev}\,\sigma_{\tau}^k:\delta \pi_{\tau}^k](\Omega\cup\Gamma_{\text{\sc d}})=\mathcal{R}(\alpha_\tau^{k-1},\delta \pi_{\tau}^k),\quad\text{with}\quad\sigma_{\tau}^k:=\C(\alpha_{\tau}^{k-1})e_{{\rm el},\tau}^k+\D(\alpha_{\tau}^{k-1})\delta e_{{\rm el},\tau}^k.
\end{equation}
\end{definition}

\begin{remark}[The discrete flow-rule]
A crucial difference with respect to the results in \cite[Proposition 3.3]{babadjian.mora} is the fact 
that condition \eqref{eq:discrete-flow-rule} is much weaker than the
differential inclusion \eqref{eq:form-disc-flow-rule}. This is due to
a key peculiarity \GGG of \EEE our model, for we consider a viscous contribution
involving only the elastic strain, but we still allow for perfect plasticity. In fact, in our setting \eqref{eq:form-disc-flow-rule} is only formal, as for every $\tau$ and $k$, the map $\delta \pi_{\tau}^k$ is a bounded Radon measure. In particular the quantity 
$\SYLD(\alpha_{\tau}^{k-1}){\rm Dir}(\delta \pi_{\tau}^k)$ does not have a pointwise meaning. As customary in the setting of perfect plasticity, the discrete flow-rule is thus only recovered in a weaker form.
\end{remark}

\subsection{Existence of weak solutions}
\UUU Let us start by specifying the discretization of the boundary
Dirichlet data as 
system \EEE 
$$
u_{\text{\sc d},{\tau}}^0:=u_{\text{\sc d}}(0),\quad u_{\text{\sc d},{\tau}}^{-1}:=u_{\text{\sc d}}(0)-\tau \DT{u}_{\text{\sc d}}(0),\quad u^k_{\text{\sc d},{\tau}}:=u_{\GGG\text{\sc d}}(k\tau)\quad\text{for every }k\in\{1,\dots,{T}/{\tau}\}.
$$
\UUU As for initial data, we recall \eqref{eq:hp} and prescribe   
\begin{align*}
u^0_{\tau}:=u_0,  \ \ \pi^0_\tau := \pi_0, \ \  \alpha_{\tau}^0 :=\alpha_0,
\ \ e^0_{\rm el,\tau} = e(u_0)-\pi_0.
\end{align*}
In order to reproduce the higher-order tests of Subsection
\ref{subs:higher-order} at the discrete level we need to specify
additionally the following  
\begin{align*}
u^{-1}_{\tau}:=u_0-\tau v_0, \ \ \GGG \pi^{-1}_\tau :=\pi_0-\tau\UUU \DT \pi_0, \ \
\alpha^{-1}_\tau := \alpha^0_\tau, \ \  e^{-1}_{\rm el,\tau} =\GGG e(u_0)-\tau(
e(v_0)-\DT \pi_0).\UUU
\end{align*} 
In particular, the last condition in \eq{eq:hp-in-data++}
ensures that  the discrete flow rule \eq{eq:form-disc-flow-rule} holds
at level $k=0$ as well. \EEE

\UUU In order to check for the solvability of the discrete system
\eqref{eq:form-disc} we proceed in two steps. For given $\alpha^{k-1}_\tau\in
W^{1,p}(\Omega)\cap L^{\infty}(\Omega)$ with $0\leq
\alpha^{k-1}_\tau\leq 1$ we look for the triple
$(u^k_{\tau},e^k_{{\rm el},{\tau}},\pi^k_{\tau})$ \UUU given by the unique solution to the minimum problem
\begin{align}\nonumber
\min&\Bigg\{
\int_{\Omega}\GGG\Big(\UUU\frac12\C(\alpha_{\tau}^{k-1})e:e
+\frac{1}{2\tau}
\D(\alpha_{\tau}^{k-1})(e-e_{{\rm el},\tau}^{k-1}):(e-e_{{\rm el},\tau}^{k-1})-f^k_{\tau}\cdot u\GGG\Big)\UUU\,\dx
\\
&
+\frac{\rho}{2\tau^2}\|u-2u_{\tau}^{k-1}+u_{\tau}^{k-2}\|^2_{L^2(\Omega;\R^d)}
+\mathcal{R}(\alpha_{\tau}^{k-1},\pi-\pi_{\tau}^{k-1})
\ : \ (u,e,\pi)\in \mathscr{A}(u_{\text{\sc d},{\tau}}^k) \Bigg\}.
\label{eq:min-disp-strain}\end{align}
where $\mathscr{A}(\cdot)$ is defined in \eqref{eq:def-A-w}.
The existence and uniqueness of solutions to
\eqref{eq:min-disp-strain} is ensured by \UUU Lemma \ref{lemma1}
below. 

Once $(u^k_{\tau},e^k_{{\rm el},{\tau}},\pi^k_{\tau})$ are found, we
determine $\alpha^k_\tau$ by solving 
 \begin{align}\nonumber
 &\min\Bigg\{\int_{\Omega}\bigg(
\tau\zeta\Big(\frac{\alpha-\alpha_{\tau}^{k-1}}{\tau}\Big)
+\frac{\kappa}{p}|\nabla \alpha|^p
\nonumber\\
&\quad+
   \int_0^{\alpha(x)}\frac12\C^\circ(s,\alpha_{\tau}^{k-1}(x))\,e^k_{{\rm el},\tau}(x):e^k_{{\rm el},\tau}(x)-\phi^\circ(s,\alpha_{\tau}^{k-1}(x))\,\d s\bigg)\,\dx \ : 
\\
 &\label{eq:f-tau-k}\qquad
   \alpha \in W^{1,p}(\Omega), \ 0\leq \alpha \leq 1\Bigg\}
 \end{align}
in Lemma \ref{lemma:damage1}
below.
\EEE

\begin{lemma}[Existence of time-discrete displacements and strains]\label{lemma1}
 Let $\alpha_{\tau}^{k-1}\in W^{1,p}(\Omega)\cap L^{\infty}(\Omega)$,
 \UUU with  $0\leq \alpha_{\tau}^{k-1}\leq 1$,  be given. \EEE Then, there exists a unique triple $(u_{\tau}^k,e_{{\rm el},\tau}^k, \pi_{\tau}^k)\in \pdeor(u^k_{\text{\sc d},\tau})$ solving \eqref{eq:min-disp-strain}.
  \end{lemma}

 \begin{proof}
 The result follows by compactness, lower-semicontinuity, and by
 Korn's inequality \eqref{eq:korn}. The uniqueness of the solution is
 a consequence of the strict convexity of the functional, and the \UUU
 fact that $\pdeor(u^k_{\text{\sc d},\tau})$ is affine. \EEE
 \end{proof}
 
Minimizers of  \eqref{eq:min-disp-strain} satisfy the following first order 
optimality conditions.

\begin{proposition}[Time-discrete Euler-Lagrange equations for displacement and strains]
 \label{prop:disp-strain}
 Let $\alpha_{\tau}^{k-1}\in W^{1,p}(\Omega)\cap L^{\infty}(\Omega)$ be a solution to \eqref{eq:form-disc-damage} satisfying $0\leq \alpha_{\tau}^{k-1}\leq 1$. Let $(u_{\tau}^k,e_{{\rm el},\tau}^k, \pi_{\tau}^k)$ be the minimizing triple of \eqref{eq:min-disp-strain}. Then, $(u_{\tau}^k,e_{{\rm el},\tau}^k, \pi_{\tau}^k)$ solves \eqref{eq:form-disc-equilibrium} and \eqref{eq:discrete-flow-rule}, ${\rm div }\,\sigma_{\tau}^k\in L^2(\Omega;\R^d)$, and $[\sigma_{\tau}^k\nu_\Gamma^{}]=0$ on $\Gamma_{\text{\sc n}}$.
\end{proposition}

\begin{proof}
We omit the proof of \eqref{eq:form-disc-equilibrium}, as it follows closely the argument in \cite[Proposition 3.3]{babadjian.mora}. The proof of \eqref{eq:discrete-flow-rule} is postponed to Corollary \ref{cor:f-r}.
\end{proof}

We conclude this subsection by showing existence of solutions to \eqref{eq:form-disc-damage}.

\begin{lemma}[Existence of admissible time-discrete damage variables]
\label{lemma:damage1}
Let $k\in\{1,\dots,{T}/{\tau}\}$, and assume that $\alpha_{\tau}^{k-1}\in 
W^{1,p}(\Omega)\cap L^{\infty}(\Omega)$, \UUU with  $0\leq
\alpha_{\tau}^{k-1}\leq 1$, is given. \EEE Let $(u_{\tau}^k,e_{{\rm el},\tau}^k,\pi_{\tau}^k)$ 
be the minimizing triple of \eqref{eq:min-disp-strain}. Then there exists 
$\alpha_{\tau}^k\in W^{1,p}(\Omega)\cap L^{\infty}(\Omega)$ solving \eqref{eq:form-disc-damage}, 
and satisfying $0\leq \alpha_{\tau}^{k}\leq 1$.
\end{lemma}

\begin{proof}
 We preliminary observe that $\alpha_{\tau}^k$ solves
 \eqref{eq:form-disc-damage} if and only if it minimizes the
 functional \UUU in \eq{eq:f-tau-k}. \EEE
The existence of a minimizer 
$\alpha_{\tau}^k\in W^{1,p}(\Omega)$ follows by the 
continuity of $\phi(\cdot)$ and $\C(\cdot)$, by lower-semicontinuity, 
and by the Dominated Convergence Theorem. The fact that $\alpha_{\tau}^k(x)\le1$ 
for every $x\in \Omega$ is a consequence of the assumption that 
$0\le\alpha_\tau^{k-1}\le1$  in $\Omega$, and \UUU of the constraint
$\alpha_{\tau}^k\leq \alpha_{\tau}^{k-1}$. The constraint
$0\leq \alpha_{\tau}^k$ instead is satisfied due to \EEE the assumptions 
on $\C(\cdot)$  and $\phi(\cdot)$ (see Subsections \ref{subs:el-visc} and 
\ref{subs:kin}), and \GGG owing \EEE to a truncation argument. 
 \end{proof}
 
\section{A-priori energy estimates}\label{sec-apriori}
\UUU In order to pass to the limit in the discrete scheme as the
fineness $\tau$ of the partition goes to $0$ we \EEE establish a few a-priori estimates on time 
interpolants between the quadruple identified via the time-discretization 
scheme \UUU of \EEE  Section \ref{sec:time-discrete}. We first rewrite 
\cite[Proposition 2.2]{dalmaso.desimone.mora06} in our framework.

\begin{lemma}[Integration by parts]
\label{prop:integ-parts}
Let $\sigma\in\Sigma(\Omega)$, $u_{\text{\sc d}}\in H^1(\Omega;\R^d)$, and 
$(u,e_{\rm el},\pi)\in\pdeor(u_{\text{\sc d}})$ with $\pdeor(\cdot)$ from 
\eq{eq:def-A-w}, with $u\in L^2(\Omega;\R^d)$. Assume that 
$[\sigma \nu_\Gamma^{}]=0$ on $\Gamma_{\text{\sc n}}$. Then
$$
[{\rm dev}\,\sigma:\pi](\Omega\cup\Gamma_{\text{\sc d}})
+\int_{\Omega}\sigma:(e_{\rm el}-e(u_{\text{\sc d}}))\,\dx
=-\int_{\Omega}{\rm div}\,\sigma\cdot(u-u_{\text{\sc d}})\,\dx.
$$
\end{lemma}
\UUU Note that the above lemma serves as definition of $[{\rm
  dev}\,\sigma:\pi](\Omega\cup\Gamma_{\text{\sc d}})$, which is a
priori not defined for ${\rm dev}\, \sigma \in L^2(\Omega;\Rdev)$ and
$\pi \in \mb$. \EEE


We are now in a position \UUU of providing, \EEE in the following lemmas and corollary, 
further optimality conditions for triples $(u_{\tau}^k,e_{\tau}^k,\pi_{\tau}^k)$ 
solving \eqref{eq:min-disp-strain}.

\begin{lemma}[Discrete Euler-Lagrange equations for the plastic strain]
\label{prop:second-E-L}
 Let $(u_{\tau}^k,e_{{\rm el},\tau}^k, \pi_{\tau}^k)$ be the minimizing triple of 
\eqref{eq:min-disp-strain}, and let $\sigma_{\tau}^k$ be the quantity defined 
in \eqref{eq:discrete-flow-rule}. Then, there holds
 \begin{equation}
 \label{eq:need-d}
 \mathcal{R}({\alpha_\tau^{k-1}},\tau\delta\pi_{\tau}^k+\pi)-\mathcal{R}({\alpha_\tau^{k-1}},\tau\delta \pi_{\tau}^k)
-[{\rm dev}\,\sigma_{\tau}^k:\pi](\Omega\cup\Gamma_{\text{\sc d}})\geq 0
 \end{equation}
for every $\pi\in \mb$ such that there exist $u\in BD(\Omega;\R^d)\cap L^2(\Omega;\R^d)$, and 
$e\in L^2(\Omega;\mthree_{\rm sym})$ with $(u,e,\pi)\in\pdeor(0)$. 
 \end{lemma}
\begin{proof}
Considering variations of the form $(u_{\tau}^k,e_{\tau}^k,\pi_{\tau}^k)+\lambda (u,e,\pi)$ for $\lambda\geq 0$ and $(u,e,\pi)\in \mathscr{A}(0)$ in \eqref{eq:min-disp-strain}, by the convexity of $\mathcal{R}$ in its second variable we obtain
$$\frac{1}{\lambda}\big(\mathcal{R}({\alpha_\tau^{k-1}},\tau\delta \pi_{\tau}^k+\lambda \pi)-\mathcal{R}({\alpha_\tau^{k-1}},\tau\delta \pi_{\tau}^k)\big)\leq \mathcal{R}({\alpha_\tau^{k-1}},\tau\delta \pi_{\tau}^k+\pi)-\mathcal{R}({\alpha_\tau^{k-1}},\tau\delta \pi_{\tau}^k),$$
which yields
\begin{align}
&\label{eq:1eq-p-extra}\int_{\Omega}\sigma_{\tau}^k:e\, \d x+ \int_{\Omega}\rho\delta^2 u_{\tau}^k\cdot u\, \d x+\mathcal{R}({\alpha_\tau^{k-1}},\tau\delta \pi_{\tau}^k+\pi)-\mathcal{R}({\alpha_\tau^{k-1}},\tau\delta \pi_{\tau}^k)-\int_{\Omega}f_{\tau}^k\cdot u\, \d x\geq 0,
\end{align}
for every $u\in BD(\Omega;\R^d)\cap L^2(\Omega;\R^d)$, $e\in L^2(\Omega;\mthree_{\rm sym})$, and $\pi\in \mb$ such that $(u,e,\pi)\in\pdeor(0)$.
In view of Lemma \ref{prop:integ-parts}, and by \eqref{eq:form-disc-equilibrium} the previous inequality implies \eqref{eq:need-d}.
\end{proof}

\begin{corollary}[Discrete flow-rule]\label{cor:f-r}
 Let $(u_{\tau}^k,e_{{\rm el},\tau}^k, \pi_{\tau}^k)$ be the minimizing triple of \eqref{eq:min-disp-strain}, let $\alpha_{\tau}^k$ be the solution to \eqref{eq:form-disc-damage} provided by Lemma \ref{lemma:damage1}, and let $\sigma_{\tau}^k$ be the quantity defined in \eqref{eq:discrete-flow-rule}. Then, $(u_{\tau}^k,e_{{\rm el},\tau}^k, \pi_{\tau}^k,\alpha_{\tau}^k)$ solve the discrete flow-rule \eqref{eq:discrete-flow-rule}.
 \end{corollary}
\begin{proof}
The assert follows by choosing $\pi=\UUU \tau\delta \pi_{\tau}^k$, \EEE and $\pi=-\tau\delta \pi_{\tau}^k$ in \eqref{eq:need-d}.
\end{proof}

\begin{lemma}
\label{prop:third-E-L}
 For $k\in \{1,\dots{T}/{\tau}\}$, let $(u_{\tau}^k,e_{{\rm el},\tau}^k, \pi_{\tau}^k)$ be the 
minimizing triple of \eqref{eq:min-disp-strain}, and let $\sigma_{\tau}^k$ be the quantity 
defined in \eqref{eq:discrete-flow-rule}. Then, there holds 
\begin{align*}\mathcal{R}({\alpha_\tau^{k-1}},\tau\delta &\pi_{\tau}^k+\pi)+
\mathcal{R}(\alpha_{\tau}^{k-\UUU 2},\tau\delta
  \pi_{\tau}^{k-1}-\pi)-\mathcal{R}({\alpha_\tau^{k-1}},\tau\delta
                         \pi_{\tau}^k)\ \
  \\&\qquad-\mathcal{R}(\alpha_{\tau}^{k-\UUU 2},\tau\delta \pi_{\tau}^{k-1})
-\tau[{\rm dev}\,\delta\sigma_{\tau}^k:\pi](\Omega\cup\Gamma_{\text{\sc d}})\geq 0
\end{align*}
for every $\pi\in \mb$ such that there exist $u\in BD(\Omega;\R^d)\cap L^2(\Omega;\R^d)$, and $e\in L^2(\Omega;\mthree_{\rm sym})$ with $(u,e,\pi)\in\pdeor(0)$. 
 \end{lemma}

\begin{proof}
Considering variations of the form $(u_{\tau}^{k-1},e_{\tau}^{k-1},\pi_{\tau}^{k-1})-\lambda (u,e,\pi)$ 
for $\lambda\geq 0$ and $(u,e,\pi)\in \mathscr{A}(0)$ in
\eqref{eq:min-disp-strain} \UUU at level $i-1$, \EEE the 
convexity of $\mathcal{R}$ in its second variable yields
\begin{align}
\label{eq:2eq-p-extra}\mathcal{R}(\alpha_{\tau}^{k-\UUU 2},\tau\delta
  \pi_{\tau}^{k-\GGG 1}\!-\pi)
-\mathcal{R}(\alpha_{\tau}^{k-1},\tau\delta \pi_{\tau}^{k-1})-\int_{\Omega}\GGG\Big(\EEE\sigma_{\tau}^{k-1}:e
+\rho\delta^2 u_{\tau}^{k-1}\cdot u
-f_{\tau}^{k-1}\cdot u\GGG\Big)\EEE\, \d x\geq 0,
\end{align}
for every $u\in BD(\Omega;\R^d)\cap L^2(\Omega;\R^d)$, $e\in L^2(\Omega;\mthree_{\rm sym})$, 
and $\pi\in \mb$ such that $(u,e,\pi)\in\pdeor(0)$. The assert follows by summing 
\eqref{eq:1eq-p-extra} and \eqref{eq:2eq-p-extra}, and by applying 
Lemma~\ref{prop:integ-parts}, and \eqref{eq:form-disc-equilibrium}.
\end{proof}

Let now $\underline{u}_{\tau}$, and $\bar{u}_{\tau}$ be the backward- and forward- piecewise constant in-time interpolants associated to the maps $u_{\tau}^k$, namely
\begin{equation}
\label{eq:b-int}
\underline{u}_{\tau}(0):=u_0,\quad \underline{u}_{\tau}(t):=u_{\tau}^{k-1}\quad\text{for every }t\in [(k-1)\tau,k\tau),\,k\in\{1,\dots,{T}/{\tau}\},
\end{equation}
and
\begin{equation}
\label{eq:f-int}
\bar{u}_{\tau}(0):=u_0,\quad \bar{u}_{\tau}(t):=u_{\tau}^k\quad\text{for every }t\in ((k-1)\tau,k\tau],\,k\in\{1,\dots,{T}/{\tau}\}.
\end{equation}
Denote by $u_{\tau}$ the associated piecewise affine in-time interpolant, that is
\begin{equation}
\label{eq:aff-int}
{u}_{\tau}(0):=u_0,\quad u_{\tau}(t):=\frac{(t-(k{-}1)\tau)}{\tau}u_{\tau}^k
+\Big(1-\frac{(t-(k{-}1)\tau)}{\tau}\Big)u_{\tau}^{k-1},
\end{equation}
for every $t\in ((k-1)\tau,k\tau],\,k\in\{1,\dots,{T}/{\tau}\}$, and let finally $\widetilde{u}_{\tau}$ be the piecewise quadratic interpolant satisfying $\widetilde{u}(k\tau)=u_{\tau}^k$, and
$$\DDT{\widetilde{u}}_{\tau}(t)=\delta^2 u_{\tau}^k\quad\text{for every }t\in ((k-1)\tau,k\tau],\,k\in\{1,\dots,{T}/{\tau}\}.$$

Let $\underline{\alpha}_{\tau},\bar{\pi}_{\tau},\bar{e}_{{\rm el},\tau}, \bar{\alpha}_{\tau},{\pi}_{\tau},e_{\tau}$, and ${\alpha}_{\tau}$ be defined analogously. The following proposition provides a first uniform estimate for the above quantities.

\begin{proposition}[Discrete energy inequality]
\label{prop:en-ineq1}
Under assumptions 
\eqref{eq:hp}, 
the following energy inequality holds true
\begin{align}
&\nonumber\int_{\Omega}\frac\rho2|\DT{u}_{\tau}(\GGG T \EEE)|^2\,\d x
+\frac\tau2\int_0^{\GGG T \EEE}\!\!\int_{\Omega}\rho|\DDT{\widetilde{u}}_{\tau}|^2\,\d x\,\d s
+\int_\tau^{\GGG T \EEE}\int_{\Omega}\rho\DT{u}_{\tau}(\cdot-\tau)\cdot\DDT{\widetilde{u}}_{{\text{\sc d},\tau}}\,\d x\,\d s\\
&\nonumber\qquad +D_{\mathcal{R}}(\alpha_{\tau};\pi_{\tau};0,\GGG T \EEE)+\int_{\Omega}\GGG\Big(\EEE\frac12\C({\underline{\alpha}}_{\tau}(\GGG T \EEE))\bar{e}_{{\rm el},\tau}(\GGG T \EEE):\bar{e}_{{\rm el},\tau}(\GGG T \EEE)-\phi(\alpha_{\tau}(\GGG T \EEE))\,\d x+\frac{\UUU \kappa}{p}|\nabla \alpha_{\tau}(\GGG T \EEE)|^p\GGG\Big)\EEE\,\d x\\
&\nonumber\qquad
+\int_0^{\GGG T \EEE}\!\!\int_{\Omega}\D({\underline{\alpha}}_{\tau})\DT{e}_{{\rm el},\tau}:\DT{e}_{{\rm el},\tau}\,\d x\,\d s+\int_{\Omega}\eta \DT{\alpha}_{\tau}(\GGG T \EEE)^2\,\d
  x\\
&\nonumber \quad\leq \int_{\Omega}\GGG\Big(\EEE\frac\rho2 v_0^2+\rho
  \DT{u}_{\tau}(\GGG T \EEE)\cdot \DT{u}_{{\text{\sc d}},\tau}(t) + \rho v_0\cdot \delta u_{{\text{\sc d}},\tau}^1\GGG\Big)\EEE\, \d x\\
&\nonumber \qquad
  +\int_{\Omega}\frac12\C(\alpha_0)(e(u_0)-\pi_0):(e(u_0)-\pi_0)
  -\phi(\alpha_0)\,\d x+ \frac{\UUU \kappa}{p}|\nabla \alpha_0|^p\,\d
  x\nonumber\\
&\qquad+\int_0^{\GGG T \EEE}\!\!\int_{\Omega}\GGG\Big(\EEE\C(\underline{\alpha}_{\tau})\bar{e}_{{\rm el},\tau}:e(\DT{u}_{{\text{\sc d}},\tau})+\D(\underline{\alpha}_{\tau})\DT{e}_{{\rm el},\tau}:e(\DT{u}_{{\text{\sc d}},\tau})+\bar{f}_{\tau}\cdot(\DT{u}_{\tau}-\DT{u}_{{\text{\sc d}},\tau})\GGG\Big)\EEE\,\d x\,\d s
.\label{eq:together1bis}
\end{align}
\end{proposition}

\begin{proof}
Fix $k\in\{1,\dots,{T}/{\tau}\}$. Testing \eqref{eq:form-disc-damage}
against $\delta \alpha_{\tau}^k$, we deduce the \UUU equality \EEE
\begin{align}
&\int_{\Omega}\eta |\delta \alpha_{\tau}^k|^2\,\d x+
  \int_{\Omega}\frac12\C^\circ(\alpha_{\tau}^k,\alpha_{\tau}^{k-1})\delta\alpha_{\tau}^k
  e^k_{{\rm el},\tau}:e^k_{{\rm el},\tau}\, \d x\nonumber\\
&\qquad +\int_{\Omega}\Big(-\UUU \phi^\circ(\alpha_{\tau}^k,
  \alpha_{\tau}^{k-1}) \EEE \delta \alpha_{\tau}^k+|\nabla
  \alpha_{\tau}^k|^{p-2}\nabla \alpha_{\tau}^k\cdot \nabla (\delta
  \alpha_{\tau}^k)\Big)\, \d x\UUU = \EEE 0. \label{eq:equality-damage1}
\end{align}
\UUU Taking \EEE $\delta u_{\tau}^k\UUU  -\delta u^k_{\text{\sc d},{\tau}}\EEE$ as test function in \eqref{eq:form-disc-equilibrium}, we have
$$ \int_{\Omega}\rho\delta^2 u_{\tau}^k \cdot (\delta u_{\tau}^k-\UUU  \delta
u^k_{\text{\sc d},{\tau}}\EEE)\,\d x-\int_{\Omega}{\rm
  div}\,\sigma_{\tau}^k\cdot (\delta u_{\tau}^k - \UUU  \delta u^k_{\text{\sc d},{\tau}}\EEE)\,\d x=\int_{\Omega}f_{\tau}^k\cdot (\delta u_{\tau}^k - \UUU  \delta u^k_{\text{\sc d},{\tau}}\EEE)\,\d x,$$
which by Lemma \ref{prop:integ-parts}, and by the fact that $[\sigma_{\tau}^k \nu_\Gamma^{}]=0$ on $\Gamma_{\text{\sc n}}$ (see Proposition \ref{prop:disp-strain}), yields
\begin{align*}
& \int_{\Omega}\rho\delta^2 u_{\tau}^k\cdot (\delta u_{\tau}^k-\delta u^k_{\text{\sc d},{\tau}})\,\d x+[{\rm dev}\,\sigma_{\tau}^k:\delta \pi_{\tau}^k](\Omega\cup\Gamma_{\text{\sc d}})+\int_{\Omega}\sigma_{\tau}^k:(\delta e_{{\rm el},\tau}^k-e(\delta u^k_{\text{\sc d},{\tau}}))\,\d x\\
&\quad =\int_{\Omega}f^k_{\tau}\cdot (\delta u_{\tau}^k-\delta u^k_{\text{\sc d},{\tau}})\,\d x.
\end{align*}
In view of Corollary \ref{cor:f-r}, we obtain 
\begin{align}\nonumber
\int_{\Omega}\rho\delta^2 u_{\tau}^k\cdot(\delta u_{\tau}^k-\delta u^k_{\text{\sc d},{\tau}})\,\d x
+\mathcal{R}({\alpha_\tau^{k-1}},\delta \pi_{\tau}^k)
+\int_{\Omega}\sigma_{\tau}^k:(\delta e_{{\rm el},\tau}^k-e(\delta u^k_{\text{\sc d},{\tau}}))\,\d x
\\\quad 
= \int_{\Omega}f^k_{\tau}\cdot (\delta u_{\tau}^k-\delta u^k_{\text{\sc d},{\tau}})\,\d x.
\label{eq:inequality-disp1}
\end{align}
For $n\in\{1,\dots,{T}/{\tau}\}$, a discrete integration by parts in time yields
\begin{align}
&
\tau\sum_{k=1}^n
\rho\delta^2 u_{\tau}^k\cdot \delta u_{\tau}^k
= \sum_{k=1}^n
\rho\big((\delta u_{\tau}^k)^2-\delta u_{\tau}^k\cdot \delta u_{\tau}^{k-1}\big)
\label{eq:add1-d}
=\frac{1}{2} 
\rho(\delta u_{\tau}^n)^2
-
\frac{1}{2}
\rho v_0^2
+\frac{\tau^2}{2}\sum_{k=1}^n 
\rho(\delta^2 u_{\tau}^k)^2
\end{align}
a.e.\ on $\Omega$. Analogously, we deduce that
\begin{align}
-\tau\sum_{k=1}^n
\rho\delta^2 u_{\tau}^k\cdot \delta u_{\text{\sc d},\tau}^k
=
\tau \sum_{k=1}^{n}
\rho\delta u_{\tau}^{k-1}\cdot \delta^2 u_{\text{\sc d},{\tau}}^{k}
-
\rho\delta u_{\tau}^n\cdot \delta u_{\text{\sc d},{\tau}}^n
\label{eq:add2-d}
- 
\rho v_0\cdot \delta u_{\text{\sc d},\tau}^0
\end{align}
a.e.\ on $\Omega$.
Additionally, by the monotonicity of $\C$ in the L\"owner ordering, and \eqref{eq:def-D}, we have
\begin{align}
&\nonumber\tau \sum_{k=1}^n \int_{\Omega}\sigma_{\tau}^k:\delta e_{{\rm el},\tau}^k\,\d x=\tau \sum_{k=1}^n
\int_{\Omega}\C(\alpha_{\tau}^{k-1})e_{{\rm el},\tau}^k:\delta e_{{\rm el},\tau}^k\,\dx+\tau \sum_{k=1}^n
\int_{\Omega}\D(\alpha_{\tau}^{k-1})\delta e_{{\rm el},\tau}^k:\delta e_{{\rm el},\tau}^k\,\dx\\
&\nonumber\quad\geq \int_{\Omega} \frac12\C(\alpha_{\tau}^{n})e_{{\rm el},\tau}^n:e_{{\rm el},\tau}^n\,\dx-\int_{\Omega} \frac12\C(\alpha_{0})(e(u_0)-\pi_0):(e(u_0)-\pi_0)\,\dx\\
&\label{eq:compute-el}\quad-\tau\sum_{k=1}^n\int_{\Omega}\frac12\delta [\C(\alpha_{\tau}^k)]e^k_{{\rm el},\tau}:e^k_{{\rm el},\tau}\,\dx+\tau \sum_{k=1}^n
\int_{\Omega}\D(\alpha_{\tau}^{k})\delta e_{{\rm el},\tau}^k:\delta e_{{\rm el},\tau}^k\,\dx,
\end{align}
and
\begin{equation}
\label{eq:compute-el2}
\frac{\tau}{2}\sum_{k=1}^n\int_{\Omega}\lineunder{\big(\C^\circ(\alpha_{\tau}^k,\alpha_{\tau}^{k-1})\delta\alpha_{\tau}^k-\delta
[\C(\alpha_{\tau}^k)]\big)}{\UUU =0}e^k_{{\rm el},\tau}:e^k_{{\rm el},\tau}\,\d
x\UUU = \EEE 0.
\end{equation}

Thus, multiplying \eqref{eq:equality-damage1} and \eqref{eq:inequality-disp1} 
by $\tau$, and summing for 
$k=1,\dots,T/\tau$, 
in view of \eqref{eq:add1-d}, \eqref{eq:add2-d}, \eqref{eq:compute-el}, and 
\eqref{eq:compute-el2} we deduce 
\begin{align}
&\nonumber\int_{\Omega}\frac\rho2|\DT{u}_{\tau}(T)|^2\,\d x+\frac{\tau}{2}\int_0^T\!\!\int_{\Omega}\rho|\DDT{\widetilde{u}}_{\tau}|^2\,\d x\,\d t
+\int_\tau^{T}\!\!\int_{\Omega}\rho \DT{u}_{\tau}(\cdot-\tau)\cdot\DDT{\widetilde{u}}_{{\text{\sc d},\tau}}\,\d x\,\d t\\
&\nonumber\qquad +\tau\sum_{k=1}^{T/\tau}
  \mathcal{R}({\alpha_\tau^{k-1}},\delta
  \pi_{\tau}^k)+\frac12\int_{\Omega}\C(\UUU {\overline{\alpha}}_{\tau}\GGG(T)\EEE)\bar{e}_{{\rm el},\tau}(\GGG T\EEE):\bar{e}_{{\rm el},\tau}(T)\,\d x\\
&\nonumber\qquad
+\int_0^T\!\!\int_{\Omega}\D({\GGG\overline{\alpha}\EEE}_{\tau})\DT{e}_{{\rm el},\tau}:\DT{e}_{{\rm el},\tau}\,\d x\,\d t+
  \int_{\Omega}\GGG\Big(\EEE\eta \DT{\alpha}_{\tau}(T)^2-\phi(\alpha_{\tau}(T))+\frac{\UUU \kappa}{p}|\nabla \alpha_{\tau}(T)|^p\GGG\Big)\EEE\,\d x\\
&\nonumber \quad\leq \int_{\Omega}\GGG\Big(\EEE\frac\rho2 v_0^2+\rho \DT{u}_{\tau}(T)\cdot \DT{u}_{{\text{\sc d}},\tau}(T)+\rho v_0\cdot \delta u_{{\text{\sc d}},\tau}^0\GGG\Big)\EEE\, \d x\\
&\nonumber \qquad
  +\int_{\Omega}\GGG\Big(\frac12\EEE\C(\alpha_0)(e(u_0)-\pi_0):(e(u_0)-\pi_0) - \phi(\alpha_0)
+
\frac{\UUU \kappa}{p}|\nabla\alpha_0|^p\GGG\Big)\EEE\,\d x\\
&\qquad\label{eq:together1}+\int_0^{\GGG T\EEE}\!\!\int_{\Omega}\C(\underline{\alpha}_{\tau})\bar{e}_{{\rm
  el},\tau}:e(\DT{u}_{{\text{\sc d}},\tau})\,\d x\,\d s\nonumber\\
&\qquad+\int_0^{\GGG T \EEE}\!\!\int_{\Omega}\D(\underline{\alpha}_{\tau})\DT{e}_{{\rm el},\tau}:e(\DT{u}_{{\text{\sc d}},\tau})
+
\bar{f}_{\tau}\cdot (\DT{u}_{\tau}-\DT{u}_{{\text{\sc d}},\tau})\,\d x\,\d s
.
\end{align}
Additionally, recalling definition 
\eqref{eq:def-a-R-diss}, and observing that $\pi_{\tau}$ jumps exactly only in 
the points $\tau k$, $k\in\{1,\dots,{T}/{\tau}\}$, by the 
monotonicity of the maps $\alpha_{\tau}$ (see Subsection \ref{sub:diss}), we have
\begin{equation}
\label{eq:together3}
D_{\mathcal{R}}(\alpha_{\tau};\pi_{\tau};0,\GGG T\EEE)=\tau\sum_{k=1}^{\GGG T/\tau} \mathcal{R}({\alpha_\tau^{k-1}},\delta \pi_{\tau}^k).
\end{equation}
This concludes the proof of the energy inequality \UUU \eq{eq:together1bis} \EEE and of the proposition.
\end{proof}

Owing to the previous discrete energy inequality, we are now in a position to deduce some first a-priori estimates for the piecewise affine interpolants.

\begin{proposition}[A-priori estimates]
\label{prop:unif-estimates}
Under assumptions \eqref{eq:hp},
\UUU for $\tau$ small enough \EEE there exists a constant $C$, dependent only on the initial conditions, on $f$,
and on $u_{\text{\sc d}}$, such that
\begin{align}
&\nonumber
\|\alpha_{\tau}\|_{H^1(0,T;L^2(\Omega))}+\|\alpha_{\tau}\|_{L^{\infty}(0,T;W^{1,p}(\Omega))}+\|e_{{\rm el},\tau}\|_{H^1(0,T;L^2(\Omega;\mthree_{\rm sym}))}\\
&\nonumber\quad+\|u_{\tau}\|_{W^{1,\infty}(0,T;L^2(\Omega;\R^d))}+\|u_{\tau}\|_{BV(0,T;BD(\Omega;\R^d))}+\|\pi_{\tau}\|_{BV(0,T;\mb)}\\
&\quad+\|\underline{\alpha}_{\tau}\|_{L^{\infty}((0,T)\times\Omega)}+\|\bar{\alpha}_{\tau}\|_{L^{\infty}((0,T)\times\Omega)}+\|\bar{e}_{{\rm el},\tau}\|_{L^{\infty}(0,T;L^2(\Omega;\mthree_{\rm sym}))}\leq C.
\label{eq:unif-estimate1}\end{align} 
\end{proposition}
\begin{proof}
The assert follows by Proposition \ref{prop:en-ineq1}, by the
regularity of the applied force $f$ and of the boundary datum
$u_{\text{\sc d}}$, and by applying H\"older's and \UUU discrete
Gronwall's  inequalities, \UUU for $\tau$ small enough. \EEE
\end{proof}
We proceed by performing at the discrete level the higher-order test with the strategy
 formally sketched in Subsection \ref{subs:higher-order}. 

\begin{proposition}[Second a-priori estimates]
\label{prop:higher-order}
Under assumptions 
\eqref{eq:hp}, \UUU for $\tau$ small enough \EEE there exists a constant $C$, dependent only on the initial conditions, on $f$, and on $u_{\text{\sc d}}$, such that
\begin{align*}
&\|\widetilde{u}_{\tau}\|_{H^2(0,T;L^2(\Omega;\R^d))}^{}
+\|{u}_{\tau}\|_{W^{1,\infty}(0,T;BD(\Omega;\R^d))}^{}\\
&\qquad\qquad+\|\pi_{\tau}\|_{W^{1,\infty}(0,T;\mb)}^{}
+\|e_{{\rm el},\tau}\|_{W^{1,\infty}(0,T;L^2(\Omega;\mthree_{\rm sym}))}^{}\leq C.
\end{align*}
\end{proposition}
\begin{proof}
Fix $k\in\{1,\dots,{T}/{\tau}\}$, and consider the map $w^k_{\tau}:=u^k_{\tau}+\taur \delta u^k_{\tau}$, where $\taur>0$ is the constant introduced in Subsection \ref{subs:el-visc}. Since $\delta^2 u^k_{\tau}=
(\delta w^k_{\tau}-\delta u^k_\tau)/{\taur}$, equation \eqref{eq:form-disc-equilibrium} rewrites as
\begin{equation}
\label{eq:E-L-w}
\rho\Big(\frac{\delta w^k_{\tau}}{\taur}\Big)-{\rm div}\,\sigma^k_{\tau}=f^k_{\tau}+\rho\Big(\frac{\delta u^k_{\tau}}{\taur}\Big).
\end{equation}
Now, testing \eqref{eq:E-L-w} against $\delta w^k_{\tau}-(\delta u_{\text{\sc d}, \tau}^k+\taur \delta^2 u_{\text{\sc d}, \tau}^k)$, by Lemma \ref{prop:integ-parts} we deduce the estimate
\begin{align}
&\nonumber\frac{1}{\taur}\int_{\Omega}\rho|\delta w^{\tau}_k|^2\,\dx+[{\rm dev}\,\sigma_{\tau}^k:(\delta \pi_{\tau}^k+\taur \delta^2 \pi_{\tau}^k)](\Omega\cup \Gamma_{\text{\sc d}})\\
&\nonumber\qquad +\int_{\Omega} \sigma_{\tau}^k:(\delta e_{{\rm el},\tau}^k+\taur \delta^2 e_{{\rm el},\tau}^k- e(\delta u_{\text{\sc d}, \tau}^k)-\taur  e( \delta^2 u_{\text{\sc d}, \tau}^k))\,\d x\\
&\nonumber \quad=\int_{\Omega} f_{\tau}^k\cdot (\delta w_{\tau}^k-\delta u_{\text{\sc d}, \tau}^k-\taur   \delta^2 u_{\text{\sc d}, \tau}^k)\,\dx+\frac{1}{\taur}\int_{\Omega}\rho\delta u_{\tau}^k\cdot (\delta w_{\tau}^k-\delta u_{\text{\sc d}, \tau}^k-\taur   \delta^2 u_{\text{\sc d}, \tau}^k)\,\dx\\
&\label{eq:2nd-a-pr-1}\qquad +\frac{1}{\taur}\int_{\Omega}\rho\delta w_{\tau}^k\cdot (\delta u_{\text{\sc d}, \tau}^k+\taur   \delta^2 u_{\text{\sc d}, \tau}^k)\,\dx.
\end{align}
In view of Lemma \ref{prop:integ-parts} we have
\begin{align*}
[{\rm dev}\,\sigma_{\tau}^k:(\delta \pi_{\tau}^k+\taur \delta^2 \pi_{\tau}^k)](\Omega\cup \Gamma_{\text{\sc d}})=[{\rm dev}\,\sigma_{\tau}^k:\delta \pi_{\tau}^k](\Omega\cup \Gamma_{\text{\sc d}})+\taur[{\rm dev}\,\sigma_{\tau}^k:\delta^2 \pi_{\tau}^k](\Omega\cup \Gamma_{\text{\sc d}}).
\end{align*}
Now, Corollary \ref{cor:f-r} yields
\begin{equation}
\label{eq:2nd-a-pr-2}[{\rm dev}\,\sigma_{\tau}^k:\delta \pi_{\tau}^k](\Omega\cup \Gamma_{\text{\sc d}})= \mathcal{R}({\alpha_\tau^{k-1}},\delta \pi_{\tau}^k),
\end{equation}
whereas Lemma \ref{prop:third-E-L} entails 
\begin{align}
&\nonumber\taur[{\rm dev}\,\sigma_{\tau}^k: \delta^2 \pi_{\tau}^k)](\Omega\cup \Gamma_{\text{\sc d}})=\taur \delta \{[{\rm dev}\,\sigma_{\tau}^k:\delta \pi_{\tau}^k] (\Omega\cup \Gamma_{\text{\sc d}})\}-\taur [{\rm dev}\,\delta \sigma_{\tau}^k:\delta \pi_{\tau}^{k-1}](\Omega\cup \Gamma_{\text{\sc d}})\\
&\nonumber\quad\geq \taur \delta \{[{\rm dev}\,\sigma_{\tau}^k:\delta
  \pi_{\tau}^k] (\Omega\cup \Gamma_{\text{\sc
  d}})\}+\mathcal{R}({\alpha_\tau^{k-1}},\delta
  \pi_{\tau}^k)+\mathcal{R}(\alpha_{\tau}^{k-\UUU 2},\delta \pi_{\tau}^{k-1})-\mathcal{R}({\alpha_\tau^{k-1}},\delta \pi_{\tau}^k+\delta \pi_{\tau}^{k-1})\\
&\nonumber\quad \geq \taur \delta \{[{\rm dev}\,\sigma_{\tau}^k:\delta \pi_{\tau}^k] (\Omega\cup \Gamma_{\text{\sc d}})\}+\mathcal{R}({\alpha_\tau^{k-1}},\delta \pi_{\tau}^k)+\mathcal{R}({\alpha_\tau^{k-1}},\delta \pi_{\tau}^{k-1})-\mathcal{R}({\alpha_\tau^{k-1}},\delta \pi_{\tau}^k+\delta \pi_{\tau}^{k-1})\\
&\label{eq:2nd-a-pr-3}\quad\geq \taur \delta \{[{\rm dev}\,\sigma_{\tau}^k:\delta \pi_{\tau}^k] (\Omega\cup \Gamma_{\text{\sc d}})\},
\end{align}
where the second-to-last step follows by the fact that $\SYLD$ is nondecreasing (see 
Subsection \ref{sub:diss}), and the last step is a consequence of the
\UUU triangle \EEE inequality. 
By combining \eqref{eq:2nd-a-pr-1}, \eqref{eq:2nd-a-pr-2}, and \eqref{eq:2nd-a-pr-3}, we obtain
\begin{align}
&\nonumber\frac{1}{\taur}\int_{\Omega}\rho|\delta w_{\tau}^k|^2\,\dx
+\taur\delta\{[{\rm dev}\,\sigma_{\tau}^k:\delta\pi_{\tau}^k](\Omega\cup \Gamma_{\text{\sc d}})\}
+\mathcal{R}({\alpha_\tau^{k-1}},\delta \pi_{\tau}^k)
\\
&\nonumber\qquad+\int_{\Omega} \sigma_{\tau}^k:(\delta e_{{\rm el},\tau}^k
+\taur \delta^2 e_{{\rm el},\tau}^k- e(\delta u_{\text{\sc d}, \tau}^k)
-\taur e(\delta^2 u_{\text{\sc d},\tau}^k))\,\d x
\\
&\nonumber \quad\leq\int_{\Omega}f_{\tau}^k\cdot(\delta w_{\tau}^k-\delta u_{\text{\sc d}, \tau}^k
-\taur\delta^2 u_{\text{\sc d}, \tau}^k)\,\dx
+\frac{1}{\taur}\int_{\Omega}
\rho\delta u_{\tau}^k\cdot(\delta w_{\tau}^k-\delta u_{\text{\sc d}, \tau}^k
-\taur\delta^2u_{\text{\sc d}, \tau}^k)\,\dx
\\
&\label{eq:2nd-a-pr-4}\qquad 
+\frac{1}{\taur}\int_{\Omega}\rho\delta w_{\tau}^k\cdot(\delta u_{\text{\sc d}, \tau}^k
+\taur\delta^2u_{\text{\sc d}, \tau}^k)\,\dx.
\end{align}
Multiplying \eqref{eq:2nd-a-pr-4} by $\tau$, summing for $k=1,\dots,n$, with 
$n\in\{1,\dots,{T}/{\tau}\}$, and using again \eqref{eq:2nd-a-pr-2} with $k=n$, 
we infer that
\begin{align}
&\nonumber\frac{\tau}{\taur}\sum_{k=1}^n\int_{\Omega}\rho|\delta
  w_{\tau}^k|^2\,\dx+\taur\mathcal{R}(\alpha_{\tau}^{\UUU n-1},\delta
  \pi_{\tau}^n)\UUU -  \taur\mathcal{R}(\alpha_{\tau}^{-1},\delta
  \pi_{\tau}^0)\EEE\\
&\nonumber\quad+\tau\sum_{k=1}^n\mathcal{R}({\alpha_\tau^{k-1}},\delta \pi_{\tau}^k)+\tau\sum_{k=1}^n\int_{\Omega} \sigma_{\tau}^k:(\delta e_{{\rm el},\tau}^k+\taur \delta^2 e_{{\rm el},\tau}^k)\,\dx\\
&\nonumber \quad\leq \tau\sum_{k=1}^n\int_{\Omega} f_{\tau}^k\cdot (\delta w_{\tau}^k-\delta u_{\text{\sc d}, \tau}^k-\taur   \delta^2 u_{\text{\sc d}, \tau}^k)\,\dx+\frac{\tau}{\taur}\sum_{k=1}^n\int_{\Omega}\rho\delta u_{\tau}^k\cdot (\delta w_{\tau}^k-\delta u_{\text{\sc d}, \tau}^k-\taur   \delta^2 u_{\text{\sc d}, \tau}^k)\,\dx\\
&\label{eq:2nd-a-pr-5} \qquad+\frac{\tau}{\taur}\sum_{k=1}^n\int_{\Omega}\rho\delta w_{\tau}^k\cdot (\delta u_{\text{\sc d}, \tau}^k+\taur   \delta^2 u_{\text{\sc d}, \tau}^k)\,\dx+\tau\sum_{k=1}^n\int_{\Omega} \sigma_{\tau}^k:(e(\delta u_{\text{\sc d}, \tau}^k)+\taur  e( \delta^2 u_{\text{\sc d}, \tau}^k))\,\dx.
\end{align}
By \eqref{eq:def-D}, 
\begin{align*}
&\tau\sum_{k=1}^n\int_{\Omega}\sigma_{\tau}^k:(\delta e_{{\rm el},\tau}^k+\taur \delta^2 e_{{\rm el},\tau}^k)\,\dx=\tau \sum_{k=1}^n\int_{\Omega}\C(\alpha_{\tau}^{k-1})(e_{{\rm el},\tau}^k+\taur \delta e_{{\rm el},\tau}^k):(\delta e_{{\rm el},\tau}^k+\taur \delta^2 e_{{\rm el},\tau}^k)\,\dx\\
&\quad+\tau\sum_{k=1}^{n}\int_{\Omega}\D_0 \delta e_{{\rm el},\tau}^k:(\delta e_{{\rm el},\tau}^k+\taur \delta^2 e_{{\rm el},\tau}^k)\,\dx.
\end{align*}
Thus, arguing as in \eqref{eq:compute-el}, we have
\begin{align}
&\nonumber\tau \sum_{k=1}^n\int_{\Omega}\sigma_{\tau}^k:(\delta e_{{\rm el},\tau}^k+\taur \delta^2 e_{{\rm el},\tau}^k)\,\dx\geq \int_{\Omega}\frac12\C(\alpha_{\tau}^{n})(e_{{\rm el},\tau}^n+\taur \delta e_{{\rm el},\tau}^n):(e_{{\rm el},\tau}^n+\taur \delta e_{{\rm el},\tau}^n)\,\dx\\
&\nonumber\quad- \int_{\Omega}
\frac12\C(\alpha_0)(e(u_0)-\pi_0+\taur
   (e(v_0)\UUU - \DT\pi_0\EEE)):(e(u_0)-\pi_0+\taur (e(v_0)\UUU -
  \DT\pi_0\EEE)
\,\dx\\&\quad
+\int_{\Omega}\Big(\tau\sum_{k=1}^n 
\D_0 \delta e_{{\rm el},\tau}^k:\delta e_{{\rm el},\tau}^k
\label{eq:add-d3}
+\frac{\taur}{2}
\D_0 \delta e_{{\rm el},\tau}^n:\delta e_{{\rm el},\tau}^n\,\d x
-\frac{\taur}{2}
\D_0 e(v_0):e(v_0)\Big)\,\d x.
\end{align}
Eventually, by \eqref{eq:2nd-a-pr-5} and \eqref{eq:add-d3}, and by
recalling \eqref{eq:together3}, for every $t\UUU = k\tau\EEE,$
\begin{align*}
&\frac{1}{\taur}\int_0^t\!\!\int_{\Omega}\rho|\DT{w}_{\tau}|^2\,\d x\,\d s
+{\taur}\mathcal{R}(\bar{\alpha}_{\tau}(t),\DT{\pi}_{\tau}(t))
+D_{\mathcal{R}}(\alpha_{\tau};\pi_{\tau};0,t)+\frac{\taur}{2}\int_{\Omega}\D_0 \DT{e}_{{\rm el},\tau}(t):\DT{e}_{{\rm el},\tau}(t)\,\d x\\
&\qquad+\int_{\Omega}\C({\underline{\alpha}}(t))(\bar{e}_{{\rm el},\tau}(t)+\taur\DT{e}_{{\rm el},\tau}(t)):(\bar{e}_{{\rm el},\tau}(t)+\taur\DT{e}_{{\rm el},\tau}(t))\,\d x+\int_0^t\!\!\int_{\Omega}\!\!\D_0 \DT{e}_{{\rm el},\tau}:\DT{e}_{{\rm el},\tau}\,\d x\,\d s\\
&\quad\leq \int_{\Omega}\Big(\frac12 \C(\alpha_0)(e(u_0){-}\pi_0{+}
\taur e(v_0){-}\taur\DT\pi_0)
{:}(e(u_0){-}\pi_0{+}\taur e(v_0){-}\taur\DT\pi_0)
+\frac{\taur}{2}
\D_0 e(v_0){:}e(v_0)\Big)\d x\\
&\qquad +\int_0^t\!\!\int_{\Omega}\bar{f}_{\tau}\cdot (\DT{w}_{\tau}-\DT{u}_{{\text{\sc d}},\tau}-\taur \DDT{\widetilde{u}}_{{\text{\sc d}},\tau})\,\d x\,\d s+\frac{1}{\taur}\int_0^t\!\!\int_{\Omega} \DT{u}_{\tau}\cdot (\DT{w}_{\tau}-\DT{u}_{{\text{\sc d}},\tau}-\taur \DDT{\widetilde{u}}_{{\text{\sc d}},\tau})\,\d x\,\d s\\
&\qquad+\frac{1}{\taur}\int_0^t\!\!\int_{\Omega} \DT{w}_{\tau}\cdot
  (\DT{u}_{{\text{\sc d}},\tau}+\taur \DDT{\widetilde{u}}_{{\text{\sc
  d}},\tau})\,\d x\,\d s+\UUU \taur\mathcal R(\alpha_{0},\DT\pi_0)\EEE\\
&\qquad+\int_0^t\!\!\int_{\Omega}\big(\C(\underline{\alpha}_{\tau})\bar{e}_{{\rm el},\tau}+\D(\underline{\alpha}_{\tau})\DT{e}_{{\rm el},\tau}\big):(e(\DT{u}_{{\text{\sc d}},\tau})+\taur e(\DDT{\widetilde{u}}_{{\text{\sc d}},\tau})))\,\d x\,\d s.
\end{align*}
The assert follows by H\"older's inequality, Proposition \ref{prop:unif-estimates}, and the assumptions on $\SYLD$ (see Subsection \ref{sub:diss}).
\end{proof}

\section{Convergence and proof of Theorem \ref{thm:main}}\label{sec:proof-main}
\begin{proposition}[Compactness]
\label{prop:compactness}
Under the assumptions of Theorem \ref{thm:main}, there exist 
\UUU $\alpha$, $e_{\rm el}$, $\pi$, and $u$ \EEE
such that $(u(t),e_{\rm el}(t),\pi(t))\in \mathscr{A}(u_{\text{\sc d}}(t))$ for every $t\in [0,T]$ (see \eqref{eq:def-A-w}), the initial conditions \eqref{initial conditions} are satisfied, and up to the extraction of a (non-relabeled) subsequence, there holds
\begin{subequations}\label{eq:comp}\begin{align}
&\label{eq:comp1}\alpha_{\tau}\wk\alpha\quad\text{weakly* in }\,H^1(0,T;L^2(\Omega))\cap L^{\infty}(0,T;W^{1,p}(\Omega)),\\
&\label{eq:comp3}e_{{\rm el},\tau}\wks e_{\rm el}\quad\text{weakly* in }\,W^{1,\infty}(0,T;L^2(\Omega;\mthree_{\rm sym})),\\
&\label{eq:comp4}\pi_{\tau}\wks \pi\quad\text{weakly* in }\,W^{1,\infty}(0,T;\mb),\\
&\label{eq:comp5}u_{\tau}\wks u\quad\text{weakly* in
  }\,W^{1,\infty}(0,T;BD(\Omega;\R^d))\cap \GGG H^1\EEE(0,T;L^2(\Omega;\R^d)),\\
&\label{eq:comp2}\underline{\alpha}_{\tau}\wks\alpha
\ \text{ and }\ \bar{\alpha}_{\tau}\wks\alpha
\quad\text{weakly* in }\,L^{\infty}((0,T)\times\Omega),
\\
&\label{eq:comp3bis}\bar{e}_{{\rm el},\tau}\wks e_{\rm el}\quad\text{weakly* in }\,L^{\infty}(0,T;L^2(\Omega;\mthree_{\rm sym})),\\
&\label{eq:comp6}\widetilde{u}_{\tau}\wk u\quad\text{weakly in }\,H^2(0,T;L^2(\Omega;\R^d)).
\end{align}\end{subequations}
\end{proposition}

\begin{proof}
Properties \eqref{eq:comp1}--\eqref{eq:comp5} are a consequence of 
Propositions~\ref{prop:unif-estimates} and \ref{prop:higher-order}. 
The admissibility condition (C1) (see Definition \ref{def:weak}) follows by the 
same argument as in \cite[Lemma 2.1]{dalmaso.desimone.mora06}. Additionally, by 
Proposition \ref{prop:unif-estimates} there holds
\begin{equation}
\label{eq:aux1}
u_{\tau}\wks u\quad\text{weakly* in }\,W^{1,\infty}(0,T;L^2(\Omega;\R^d)),
\end{equation}
and there exist $\check{\alpha},\hat{\alpha}\in L^{\infty}((0,T)\times\Omega)$, and 
$\hat{e}\in L^{\infty}(0,T;L^2(\Omega;\mthree_{\rm sym}))$ such that
\begin{equation*}
\bar{\alpha}_{\tau}\wks\check{\alpha}\ \text{ and }\ 
\underline{\alpha}_{\tau}\wks\hat{\alpha}
\quad\text{weakly* in }\,L^{\infty}((0,T)\times\Omega)
\end{equation*}
and \begin{equation*}
\bar{e}_{{\rm el},\tau}\wks \hat{e}\quad\text{weakly* in }\,L^{\infty}(0,T;L^2(\Omega;\mthree_{\rm sym})).
\end{equation*}
Additionally by Proposition \ref{prop:higher-order} there exists a map 
$\hat{u}\in H^2(0,T;L^2(\Omega;\R^d))$ such that, up to the extraction of 
a (non-relabeled) subsequence,
\begin{equation}
\label{eq:aux2}
\widetilde{u}_{\tau}\wk \hat{u}\quad\text{weakly in }\,H^2(0,T;L^2(\Omega;\R^d)).
\end{equation}
By the compact embeddings of $W^{1,\infty}(0,T;L^2(\Omega;\R^d))$ and 
$H^2(0,T;L^2(\Omega;\R^d))$ into $C_{\rm w}(0,T;L^2(\Omega;\R^d))$, 
we deduce
\begin{equation}
\label{eq:aux3}
u_{\tau}(t)\wk u(t)\quad\text{weakly in }\,L^2(\Omega;\R^d),
\end{equation} 
and 
\begin{equation}
\label{eq:aux4}
\widetilde{u}_{\tau}(t)\wk \hat{u}(t)\quad\text{weakly in }\,L^2(\Omega;\R^d),
\end{equation} 
for every $t\in [0,T]$.
To complete the proof of \GGG\eqref{eq:comp} \EEE, it remains 
to show that $\check{\alpha}=\hat{\alpha}=\alpha$, $\hat{e}=e_{\rm el}$, and $\hat{u}=u$. 

We proceed by showing this last equality; the proof of the other two identities is analogous. 
Fix $k\in \{1,\dots,{T}/{\tau}\}$, and $t\in ((k{-}1)\tau,k\tau]$. Then, using the fact that
$$\DT{\widetilde{u}}_{\tau}(t)=\frac{(t-(k{-}1)\tau)}{\tau}\delta u_{\tau}^k+\Big(1-\frac{(t-(k{-}1)\tau)}{\tau}\Big)\delta u_{\tau}^{k-1}$$
for every $t\in ((k{-}1)\tau,k\tau]$, we have
\begin{align}
&\UUU
  \int_0^T\|\DT {\widetilde{u}}_{\tau}(t)-\DT
  u_{\tau}(t)\|^2_{L^2(\Omega;\R^d)}\dt=
  \sum_{k=1}^N
  \int_{(k-1)\tau}^{k\tau}\|\DT{\widetilde{u}}_{\tau}(t)-\DT
  u_{\tau}(t)\|^2_{L^2(\Omega;\R^d)}\dt \nonumber\\
&\quad\UUU = \tau^2 \sum_{k=1}^N
  \int_{(k-1)\tau}^{k\tau} (1-\overline \alpha_\tau(t))^2 \, \dt
  \left\| \frac{\DT u_\tau - \DT u_\tau(\cdot - \tau)}{\tau}\right\|^2
  = \frac{\tau^2}{3}\sum_{k=1}^N\tau \| \delta^2
  u_k\|^2_{L^2(\Omega;\R^d)}\leq C\tau^2, \label{eq:aux5}
\end{align}
where the last inequality follows by Proposition \ref{prop:higher-order}. The assert follows 
then by combining \eqref{eq:aux3}, \eqref{eq:aux4}, and \eqref{eq:aux5}.
\end{proof}

\begin{proposition}[Strong convergence of the elastic strains]
\label{prop:strong-convergence}
Let $e_{\rm el}$ be the map identified in Proposition \ref{prop:compactness}. Under the assumptions 
of Theorem \ref{thm:main}, there holds
\begin{equation}
\label{eq:strong-el}
e_{{\rm el},\tau}\to e_{\rm el}\quad\text{strongly in }H^1(0,T;L^2(\Omega;\mthree_{\rm sym})),
\end{equation}
and
\begin{equation}
\label{eq:strong-el2}
\bar{e}_{{\rm el},\tau}(t)\to e_{\rm el}(t)\quad\text{strongly in }L^2(\Omega;\mthree_{\rm sym}))\quad\text{for a.e.}\,t\in [0,T].
\end{equation}
\end{proposition}

\begin{proof}
For $k\in\{1,\dots,{T}/{\tau}\}$, denote by $\delta e_{\rm el}(k\tau)$ the quantity
$$\delta e_{\rm el}(k\tau):=\frac{e_{\rm el}(k\tau)-e_{\rm el}((k{-}1)\tau)}{\tau},$$
and by $\bar{e}_{\rm el}^{\tau},\,{e}_{\rm el}^{\tau}$ the forward-piecewise constant and the affine interpolants between the values $\{e(k\tau)\}_{k=1,\dots,T/\tau}$ (see \eqref{eq:b-int} and \eqref{eq:aff-int}). Let $\delta u(k\tau)$, $\delta \pi(k\tau)$, $\bar{u}^{\tau},\,u^{\tau},\,\bar{\pi}^{\tau}$, and $\pi^{\tau}$ be defined 
analogously. \UUU Note that here we cannot directly use the values at
time $t$, for this would prevent relation \eqref{eq:starstarstar-strong} to
hold. Here, the pointwise value of $\pi$ is simply that of its
right-continuous representative.\EEE 

Fix $k\in\{1,\dots,{T}/{\tau}\}$. We proceed by testing the time-discrete equilibrium equation \eqref{eq:form-disc-equilibrium} by $\delta u_{\tau}^k-\delta u(k\tau)$. On the one hand, by Lemma \ref{prop:integ-parts}, we have
\begin{align}
&\nonumber\int_{\Omega}\rho \delta^2 u^k_{\tau}\cdot (\delta u^k_{\tau}-\delta u(k\tau))\,\dx+[{\rm dev}\,\sigma_{\tau}^k:(\delta \pi_{\tau}^k-\delta \pi(k\tau))](\Omega\cup \Gamma_{\text{\sc{D}}})\\
&\quad+\int_{\Omega}\sigma_{\tau}^k:(\delta e_{{\rm el},\tau}^k-\delta e_{\rm el}(k\tau))\,\d x-\int_{\Omega}f_{\tau}^k\cdot (\delta u_{\tau}^k-\delta u(k\tau))\,\d x=0.
\label{eq:1-s-c}
\end{align}
On the other hand, Lemma \ref{prop:second-E-L} yields
\begin{equation}
\label{eq:2-s-c}
[{\rm dev}\,\sigma_{\tau}^k:(\delta \pi_{\tau}^k-\delta \pi(k\tau))](\Omega\cup \Gamma_{\text{\sc{D}}})\geq \mathcal{R}({\alpha_\tau^{k-1}},\delta \pi_{\tau}^k)-\mathcal{R}({\alpha_\tau^{k-1}},\delta \pi(k\tau)).
\end{equation}
By combining \eqref{eq:1-s-c} and \eqref{eq:2-s-c}, we obtain
\begin{align}
\int_{\Omega} \sigma_{\tau}^k:(\delta e_{{\rm el},\tau}^k-\delta e(k\tau))\,\d x
&\le\mathcal{R}({\alpha_\tau^{k-1}},\delta \pi(k\tau))
-\mathcal{R}({\alpha_\tau^{k-1}},\delta \pi_{\tau}^k)+\int_{\Omega}(f_{\tau}^k-\rho \delta^2 u_{\tau}^k)\cdot (\delta u_{\tau}^k-\delta u(k\tau))\,\d x.
\label{eq:star-strong}\end{align}

In view of \UUU the definition of $\sigma_k$ \EEE there holds
\begin{align}
&\nonumber\int_{\Omega} \sigma_{\tau}^k:(\delta e_{{\rm el},\tau}^k-\delta e_{\rm el}(k\tau))\,\d x=\int_{\Omega}\C(\alpha_{\tau}^{k-1})(e_{{\rm el},\tau}^k-{e}_{\rm el}(k\tau)):(\delta e_{{\rm el},\tau}^k-\delta e_{\rm el}(k\tau))\,\d x\\
&\nonumber\quad
+\int_{\Omega}\!\D(\alpha_{\tau}^{\UUU k-1})(\delta e_{{\rm el},\tau}^k\!-\delta e_{\rm el}(k\tau))
{:}(\delta e_{{\rm el},\tau}^k\!{-}\delta e_{\rm el}(k\tau))
+
\D(\alpha_{\tau}^{\UUU k-1})\delta e_{\rm el}(k\tau){:}(\delta e_{{\rm el},\tau}^k\!{-}\delta e_{\rm el}(k\tau))\,\d x\\
&\nonumber\quad
+\int_{\Omega}\C(\alpha(k\tau))e_{\rm el}(k\tau):(\delta e_{{\rm el},\tau}^k-\delta e_{\rm el}(k\tau))\,\dx\\
&\quad\label{eq:starstar-strong}-\int_{\Omega}(\C(\alpha(k\tau))-\C(\alpha_{\tau}^{k-1}))e_{\rm el}(k\tau):(\delta e_{{\rm el},\tau}^k-\delta e_{\rm el}(k\tau))\,\d x.
\end{align}
Let now $n\in \{1,\dots,{T}/{\tau}\}$. By the monotonicity of $\C$ in the 
L\"owner order, arguing as in the proof of \eqref{eq:compute-el}, we deduce
\begin{align}
&\nonumber
\tau\sum_{k=1}^n\int_{\Omega}\C(\alpha_{\tau}^{k-1})(e^k_{{\rm el},\tau}-e_{\rm el}(k\tau)):(\delta e^k_{{\rm el},\tau}-\delta e_{\rm el}(k\tau))\,\d x
\\[-.3em]&\qquad\qquad
\label{eq:starstarstar-strong}\geq  \int_{\Omega} 
\frac12\C(\alpha_{\tau}^{n})(e^n_{{\rm el},\tau}-e_{\rm el}(n\tau)):(e^n_{{\rm el},\tau}-\delta e_{\rm el}(n\tau))\,\d x.
\end{align}
Multiplying \eqref{eq:star-strong} by $\tau$, and summing for $k=1,\dots,\GGG T/\tau\EEE$, in view of \eqref{eq:starstar-strong} and \eqref{eq:starstarstar-strong}, we obtain the estimate
\begin{align*}
& \int_{\Omega}\frac12\C({\underline{\alpha}}_{\tau}(\GGG T \EEE))(\bar{e}_{{\rm el},\tau}(\GGG T \EEE)-\bar{e}_{\rm el}^{\tau}(\GGG T \EEE)):(\bar{e}_{{\rm el},\tau}(\GGG T \EEE)-\bar{e}_{\rm el}^{\tau}(\GGG T \EEE))\, \d x
\\
&\qquad+\int_0^{\GGG T \EEE}\!\!\int_{\Omega}\D({\underline{\alpha}}_{\tau})(\DT{e}_{{\rm el},\tau}-\DT{e}_{\rm el}^{\tau}):(\DT{e}_{{\rm el},\tau}-\DT{e}_{\rm el}^{\tau})\,\d x\,\d s\\
&\quad\leq \tau\sum_{k=1}^{\GGG T/\tau \EEE} \mathcal{R}({\alpha_\tau^{k-1}},\delta \pi(k\tau))-D_{\mathcal{R}}(\alpha_{\tau};\pi_{\tau};0,\GGG T \EEE)+\int_0^{\GGG T \EEE}\!\!\int_{\Omega}(\bar{f}_{\tau}{-\DDT{\widetilde{u}}_{\tau}})\cdot (\DT{u}_{\tau}-{\DT u}^{\tau})\,\d x\,\d s\\
&\qquad-\int_0^{\GGG T \EEE}\!\!\int_{\Omega}\D({\underline{\alpha}}_{\tau})\DT{e}_{\rm
  el}^{\tau}:(\DT{e}_{{\rm el},\tau}-\DT{e}_{\rm el}^{\tau})\,\d x\,\d
  s-\int_0^{\GGG T \EEE}\!\!\int_{\Omega}\C({\UUU \overline{\alpha}}^{\tau}\EEE)\bar{e}_{\rm el}^{\tau}:(\DT{e}_{{\rm el},\tau}-\DT{e}_{\rm el}^{\tau})\,\d x\,\d s\\
&\qquad+\int_0^{\GGG T \EEE}\!\!\int_{\Omega}\big(\C({\UUU \overline{\alpha}}^{\tau}\EEE)-\C({\underline{\alpha}}_{\tau}){\big)}\bar{e}^{\tau}_{\rm el}:(\DT{e}_{{\rm el},\tau}-\DT{e}_{\rm el}^{\tau})\big)\,\d x\,\d s.
\end{align*}
By Proposition \ref{prop:compactness} we infer that
\begin{align*}
&\limsup_{\tau\to 0} \int_0^{\GGG T\EEE}\!\!\int_{\Omega}\D({\UUU
  \overline{\alpha}}_{\tau} \EEE)(\DT{e}_{{\rm el},\tau}-\DT{e}_{\rm el}):(\DT{e}_{{\rm el},\tau}-\DT{e}_{\rm el})\,\d x\,\d s\\
&\quad+\limsup_{\tau\to 0}\bigg\{\int_0^{\GGG T \EEE}\mathcal{R}(\UUU \underline
  {\alpha}_{\tau} \EEE,\DT{\pi}^{\tau})\,\d s-D_{\mathcal{R}}(\alpha_{\tau};\pi_{\tau};0,\GGG T \EEE)+\int_0^{\GGG T \EEE}\!\!\int_{\Omega}(\bar{f}_{\tau}{-\DDT{\widetilde{u}}_{\tau}})\cdot (\DT{u}_{\tau}-{\DT u}^{\tau})\,\d x\,\d s\\
&\qquad\qquad-\int_0^{\GGG T \EEE}\!\!\int_{\Omega}\D({\underline{\alpha}}_{\tau})\DT{e}_{\rm el}^{\tau}:(\DT{e}_{{\rm el},\tau}-\DT{e}_{\rm el}^{\tau})\,\d x\,\d s-\int_0^{\GGG T \EEE}\!\!\int_{\Omega}\C({\underline{\alpha}}^{\tau})\bar{e}_{\rm el}^{\tau}:(\DT{e}_{{\rm el},\tau}-\DT{e}_{\rm el}^{\tau})\,\d x\,\d s\\
&\qquad\qquad+\int_0^{\GGG T \EEE}\!\!\int_{\Omega}\big(\C({\UUU \bar{\alpha}}^{\tau}\EEE)-\C({\underline{\alpha}}_{\tau}){\big)}\bar{e}^{\tau}_{\rm el}:(\DT{e}_{{\rm el},\tau}-\DT{e}_{\rm el}^{\tau})\big)\,\d x\,\d s\bigg\}.
\end{align*}
Since $u\in H^2(0,T;L^2(\Omega;\R^d))$ and $e_{\rm el}\in W^{1,\infty}(0,T;L^2(\Omega;\mthree_{\rm sym}))$, it follows that
\begin{equation}
\label{eq:num1-strong}
u^{\tau}\to u\quad\text{strongly in }\,L^2((0,T)\times \Omega;\R^d),
\end{equation}
and
\begin{equation}
\label{eq:num2-strong}
\bar{e}_{\rm el}^{\tau}\to \bar{e}_{\rm el}\quad\text{strongly in }\,L^2((0,T)\times \Omega;\mthree_{\rm sym}).
\end{equation}
Additionally, by the definition of the affine interpolants, 
\begin{align}
&\label{eq:num2bis-strong}
\DT{e}_{\rm el}^{\tau}\to \DT{e}_{\rm el}\quad\text{strongly in }\,L^2((0,T)\times \Omega;\mthree_{\rm sym}),\\
&\label{eq:num2ter-strong}
\DT{\pi}^{\tau}\to \DT{\pi}\quad\text{strongly in }\,L^1(0,T;\mb).
\end{align}
By \eqref{eq:comp1} and by \UUU the \EEE Aubin-Lions Lemma, 
up to the extraction of a (non-relabeled) subsequence,
\begin{equation}
\label{eq:num3-strong}
\alpha_{\tau}\to \alpha\quad\text{strongly in }
{C([0,T]\times\bar\Omega)}.
\end{equation}
Since $\alpha\in H^1(0,T;L^2(\Omega))\cap L^{\infty}(0,T;W^{1,p}(\Omega))$, 
\begin{equation}
\label{eq:num4-strong}
\bar{\alpha}^{\tau}, \, \UUU \underline{\alpha}^\tau \EEE \to \alpha\quad\text{strongly in }\,L^2((0,T)\times \Omega).
\end{equation}
Thus, by the Dominated Convergence Theorem, 
we deduce that
\begin{align}
&\label{eq:num5-strong} \C({\UUU\bar{\alpha}}^{\tau}\EEE)\bar{e}^{\tau}_{\rm el}\to \C(\alpha)e_{\rm el}\quad\text{strongly in }\,L^2((0,T)\times \Omega;\mthree_{\rm sym}),\\
&\label{eq:num6-strong} \big(\C({\UUU \bar{\alpha}}^{\tau}\EEE)-\C({\underline{\alpha}}_{\tau})\big)\bar{e}^{\tau}_{\rm el}\to 0\quad\text{strongly in }\,L^2((0,T)\times \Omega;\mthree_{\rm sym}),\\
&\label{eq:num7-strong} \D({\underline{\alpha}}_{\tau})\DT{e}_{\rm el}^{\tau}\to \D(\alpha)\DT{e}_{\rm el}\quad\text{strongly in }\,L^2((0,T)\times \Omega;\mthree_{\rm sym}).
\end{align}
Finally, by the assumptions on $f$, we have
\begin{equation}
\label{eq:num8-strong}
\bar{f}_{\tau}\to f\quad\text{strongly in }\,L^2((0,T)\times \Omega;\R^d).
\end{equation}
By combining \eqref{eq:num1-strong}--\eqref{eq:num8-strong} we conclude that 
\begin{align}
&\nonumber\limsup_{\tau\to 0} \int_0^{\GGG T \EEE}\!\!\int_{\Omega}\D({\underline{\alpha}}_{\tau})(\DT{e}_{{\rm el},\tau}-\DT{e}_{\rm el}):(\DT{e}_{{\rm el},\tau}-\DT{e}_{\rm el})\,\d x\,\d s\\
&\quad\label{eq:a-f-limsup}\leq \limsup_{\tau\to 0}\int_0^{\GGG T \EEE}\mathcal{R}(\bar{\alpha}_{\tau},\DT{\pi}^{\tau})\,\d s-\liminf_{\tau\to 0}D_{\mathcal{R}}(\alpha_{\tau};\pi_{\tau};0,\GGG T \EEE).
\end{align}
Arguing as in \cite[Theorem 7.1]{dalmaso.desimone.mora06}, since $\pi\in W^{1,\infty}(0,T;\mb)$ we deduce the uniform bound
\begin{equation}
\label{eq:num9-strong}
\int_0^{\UUU T}\|\DT{\pi}^{\tau}\|_{\mb}=\tau\sum_{k=1}^{\UUU
  T/\tau}\|\delta \pi(k\tau)\|_{\mb}\leq \int_0^{\UUU T }\!\!\|\DT{\pi}\|_{\mb}\,\ds\leq C.
\end{equation}
Hence, by \eqref{eq:num9-strong} and by the continuity and monotonicity of $\SYLD(\cdot)$ (see Subsection \ref{sub:diss}), there holds
\begin{align}
&\nonumber\limsup_{\tau\to 0}\int_0^{\GGG T \EEE} \mathcal{R}(\bar{\alpha}_{\tau},\DT{\pi}^{\tau})\,\d s\leq \limsup_{\tau\to 0}\int_0^{\GGG T \EEE} \mathcal{R}(\alpha_{\tau},\DT{\pi}^{\tau})\,\d s
\\&\ \nonumber
\leq \limsup_{\tau\to 0}\Big\{\int_0^{\GGG T \EEE} \mathcal{R}(\alpha,\DT{\pi}^{\tau})\,\d s+\Big|\int_0^{\GGG T \EEE}\!(\mathcal{R}(\alpha_{\tau},\DT{\pi}^{\tau})-\mathcal{R}(\alpha,\DT{\pi}^{\tau}))\,\d s\Big|\Big\}
\\&\label{eq:num10-strong}\ 
\leq \limsup_{\tau\to 0}\int_0^{\GGG T \EEE}\!\!\!\mathcal{R}(\alpha,\DT{\pi}^{\tau})\,\d s+C\limsup_{\tau\to 0}\|\SYLD(\alpha_{\tau}){-}\SYLD(\alpha)\|_{L^{\infty}((0,T)\times \Omega)}^{}=\int_0^{\GGG T \EEE}\!\!\!\mathcal{R}(\alpha,\DT{\pi})\,\d s,
\end{align}
where the last step follows by \eqref{eq:num2ter-strong}.

To complete the proof of \eqref{eq:strong-el}, it remains to show that
\begin{equation}
\label{eq:liminf-dissipation}
D_{\mathcal{R}}(\alpha;\pi;0,\GGG T\EEE)\leq \liminf_{\tau\to 0}D_{\mathcal{R}}(\alpha_{\tau};\pi_{\tau};0,\GGG T \EEE).
\end{equation}
Let $0<t_0<t_1<\dots<t_n\leq \GGG T \EEE$. By the definition of $D_{\mathcal{R}}$, we have
\begin{align*}
&D_{\mathcal{R}}(\alpha_{\tau};\pi_{\tau};0,\GGG T \EEE)\geq \sum_{j=1}^{\GGG T/\tau \EEE}\mathcal{R}(\alpha_{\tau}(t_j),\pi_{\tau}(t_j)-\pi_{\tau}(t_{j-1}))\geq \sum_{j=1}^{\GGG T/\tau \EEE}\mathcal{R}(\alpha(t_j),\pi_{\tau}(t_j)-\pi_{\tau}(t_{j-1}))\\[-.6em]
&\qquad\qquad\qquad\qquad\ \ 
-{\tau}\sum_{j=1}^{\GGG T/\tau \EEE} \|\SYLD(\alpha_{\tau}(t_j))-\SYLD(\alpha(t_j))\|_{L^{\infty}(\Omega)}\|\DT{\pi}_{\tau}\|_{L^{\infty}(0,T;\mb)}.
\end{align*}
Now, by \eqref{eq:comp4} and \eqref{eq:num3-strong},
\begin{align*}
\lim_{\tau\to 0} 
{\tau}\sum_{j=1}^{\GGG T/\tau \EEE} \|\SYLD(\alpha_{\tau}(t_j))-\SYLD(\alpha(t_j))\|_{L^{\infty}(\Omega)}\|\DT{\pi}_{\tau}\|_{L^{\infty}(0,T;\mb)}=0.
\end{align*}
Thus, by \eqref{eq:comp4},
\begin{align*}
\!\liminf_{\tau\to 0}D_{\mathcal{R}}(\alpha_{\tau};\pi_{\tau};0,\GGG T \EEE)
\geq \liminf_{\tau\to 0}\sum_{j=1}^{\GGG T/\tau \EEE}
\mathcal{R}(\alpha(t_j),\pi_{\tau}(t_j){-}\pi_{\tau}(t_{j-1}))
\geq \sum_{j=1}^{\GGG T/\tau \EEE}\mathcal{R}(\alpha(t_j),\pi(t_j){-}\pi(t_{j-1})).
\end{align*}
By the arbitrariness of the partition, we deduce \eqref{eq:liminf-dissipation}, which in turn yields \eqref{eq:strong-el}.

Property \eqref{eq:strong-el2} follows arguing exactly as in the proof of 
\eqref{eq:aux5}.
\end{proof}


{Let us now \UUU conclude \EEE the proof of Theorem \ref{thm:main}.}

\begin{proof}[Proof of Theorem \ref{thm:main}]
Let $(u,e_{\rm el},\pi,\alpha)$ be the limit quadruple identified in 
Proposition \ref{prop:compactness}. 
By Proposition~\ref{prop:compactness} we already know that condition (C1) in 
Definition \ref{def:weak} is fulfilled. For convenience of the reader we subdivide 
the proof of the remaining conditions into three steps.
\\
\noindent \textbf{Step 1}: We first show that the limit quadruple satisfies the equilibrium equation \eqref{system-u-new}.
In view of \eqref{eq:form-disc-equilibrium} we have
$$
\rho\DDT{\widetilde{u}}_{\tau}
-{\rm div}\,(\C(\underline{\alpha}_{\tau})\bar{e}_{{\rm el},\tau}
+\D(\underline{\alpha}_{\tau})\DT{e}_{{\rm el},\tau})=\bar{f}_{\tau}
$$
for a.e. $x\in \Omega$ and $t\in [0,T]$, and for all $\tau>0$. In particular, for all 
$\varphi\in C^{\infty}_c(0,T;C^{\infty}_c(\Omega))$ there holds
\begin{align*}
&\int_0^T\!\!\int_{\Omega}\rho\DDT{\widetilde{u}}_{\tau}\cdot\varphi\,
+
(\C(\underline{\alpha}_{\tau})\bar{e}_{{\rm el},\tau}
+\D(\underline{\alpha}_{\tau})\DT{e}_{{\rm el},\tau}):{e(\varphi)}\,\d x\,\d t
=\int_0^T\!\!\int_{\Omega}\bar{f}_{\tau}\cdot \varphi\,\dx\,\d t.
\end{align*}
By 
(\ref{eq:comp}e-g)
 and \eqref{eq:num8-strong}, we infer that
\begin{align*}
\int_0^T\!\!\int_{\Omega}\rho\DDT{{u}}\cdot\varphi
+
(\C(\alpha) e_{{\rm el}}+\D(\alpha)\DT{e}_{{\rm el}}):{e(\varphi)}\,\d x\,\dt
=\int_0^T\!\!\int_{\Omega}f\cdot \varphi\,\dx\,\dt,
\end{align*}
which in turn yields \eqref{system-u-new} for a.e. $x\in \Omega$, and $t\in [0,T]$. In particular, \eqref{eq:comp6} guarantees that $u(0)=u^0$, and $\DT{u}(0)=v_0$.\\

\noindent\textbf{Step 2}: The limit energy inequality is a direct 
consequence of \eqref{eq:together1bis}, Propositions \ref{prop:compactness} 
and \ref{prop:strong-convergence}, \GGG and \eqref{eq:liminf-dissipation}. \EEE

\noindent\textbf{Step 3}: We now pass to the limit in the discrete damage law. 
In view of \eqref{eq:form-disc-damage}, for every $k\in\{1,\dots,{T}/{\tau}\}$ 
we deduce the inequality
\begin{align*}
\int_{\Omega}\Big(\UUU \phi^\circ(\alpha_{\tau}^k,
  \alpha_{\tau}^{k-1})+{\rm div}\, (|\nabla\alpha_{\tau}^k|^{p-2}\nabla\alpha_{\tau}^k)\GGG
-\frac12\C^{\circ}(\alpha^{k}_\tau,\alpha^{k-1}_\tau)\EEE e_{{\rm el},\tau}^k:e_{{\rm el},\tau}^k-\eta \delta \alpha_{\tau}^k\Big)
(\varphi-\delta \alpha_{\tau}^k)\,\d x\UUU = \EEE 0
\end{align*}
for all $\varphi\in W^{1,p}(\Omega)$ such that $\varphi(x)\le0$ for a.e.\ 
$x\in\Omega$. Thus, summing in $k$, we conclude that 
\begin{align*}
&\int_0^{\GGG T \EEE}\!\!\int_{\Omega} \Big(
{\phi^\circ(\bar\alpha_\tau,\underline\alpha_\tau)}\varphi
-|\nabla\bar{\alpha}_{\tau}|^{p-2}\nabla \bar{\alpha}_{\tau}\cdot \nabla \varphi
\GGG-\frac12 \EEE\C^\circ(\bar\alpha_\tau,\underline\alpha_\tau)\bar{e}_{{\rm el},\tau}:\bar{e}_{{\rm el},\tau}\varphi
-\eta \DT{\alpha}_{\tau}\varphi\Big)\,\dx\,\d t\\
&\quad\leq \int_{\Omega}\Big(\UUU\phi\EEE(\alpha_{\tau}(\GGG T \EEE))
-
\UUU \phi(\alpha_0)\EEE
-
\frac{\UUU \kappa}{p}|\nabla \UUU \bar\alpha_{\tau} \EEE(\GGG T \EEE)|^p
+
\frac{\UUU \kappa}{p}|\nabla \alpha_0|^p\Big)\,\d x\\[-.3em]
&\qquad\ 
\GGG-\EEE\int_0^{\GGG T \EEE}\!\!\int_{\Omega}\GGG\frac12\EEE\Big(\C^\circ(\bar\alpha_\tau,\underline\alpha_\tau)\bar{e}_{{\rm el},\tau}:\bar{e}_{{\rm el},\tau}\Big)\DT{\alpha}_{\tau}\,\d x\,\d t-\int_0^{\GGG T \EEE}\!\!\int_{\Omega}\eta\DT{\alpha}_{\tau}^2\,\d x\,\d t\,.
\end{align*}
Condition (C4) in Definition \ref{def:weak} follows then in view of Propositions \ref{prop:compactness} and \ref{prop:strong-convergence}.
\end{proof}

\bigskip

\noindent
{\it Acknowledgments}: 
This research received a partial support from 
from the Vienna Science and Technology Fund (WWTF)
project MA14-009, and from the Austrian Science Fund (FWF)
projects F\,65 and P\,27052, from the Czech Sci.\ Foundation (CSF)
projects 17-04301S and 18-03834S, from the Austrian-Czech project 
I\,2375/16-34894L (FWF/CSF), as well as
from the institutional support RVO:61388998 (\v CR).
\bibliographystyle{plain}

\providecommand{\WileyBibTextsc}{}
\let\textsc\WileyBibTextsc
\providecommand{\othercit}{}
\providecommand{\jr}[1]{#1}
\providecommand{\etal}{~et~al.}

\end{document}